\documentclass[preprint,bj]{imsart}

\usepackage{amsfonts, amsmath, bbm, enumitem, parskip, mathtools, tikz, pgfplots, xr}

\usepackage[square, numbers]{natbib}

\allowdisplaybreaks

\DeclareMathOperator*{\argmin}{arg\,min}

\DeclareMathOperator{\bor}{\mathcal{B}}

\DeclareMathOperator{\diag}{\mathsf{diag}}
\DeclareMathOperator{\expect}{\mathbb{E}}

\DeclareMathOperator{\prob}{\mathbb{P}}
\DeclareMathOperator{\rk}{rk}

\DeclareMathOperator{\spn}{sp}

\DeclareMathOperator{\trans}{\mathsf{T}}

\newcommand{\eps}{\varepsilon}
\newcommand{\iid}{i.i.d.\@ }

\newcommand{\real}{\mathbb{R}}
\newcommand{\subs}{\subseteq}

\newcommand{\BlackBox}{\rule{1.5ex}{1.5ex}} % end of proof
\newenvironment{proof}{\par\noindent{\bf Proof\ }}{\hfill\BlackBox\\[2mm]}
 
\newtheorem{theorem}{Theorem}
\newtheorem{lemma}[theorem]{Lemma}

\newtheorem{corollary}[theorem]{Corollary}

\begin{document}

\begin{frontmatter}
\title{The Goldenshluger--Lepski Method for Constrained Least-Squares Estimators over RKHSs}
\runtitle{Lepski for Constrained Estimators over RKHSs}

\begin{aug}
\author{\fnms{Stephen} \snm{Page} \thanksref{a1,e1} \ead[label=e1,mark]{s.page@lancaster.ac.uk}}
\and
\author{\fnms{Steffen} \snm{Gr\"{u}new\"{a}lder} \thanksref{a2} \ead[label=e2]{s.grunewalder@lancaster.ac.uk}}
\runauthor{Stephen Page and Steffen Gr\"{u}new\"{a}lder}
\affiliation{Lancaster University}
\address[a1]{STOR-i, Lancaster University, Lancaster, LA1 4YF, United Kingdom. \printead{e1}}
\address[a2]{Department of Mathematics and Statistics, Lancaster University, Lancaster, LA1 4YF, United Kingdom. \printead{e2}}
\end{aug}

\begin{abstract}
We study an adaptive estimation procedure called the Goldenshluger--Lepski method in the context of reproducing kernel Hilbert space (RKHS) regression. Adaptive estimation provides a way of selecting tuning parameters for statistical estimators using only the available data. This allows us to perform estimation without making strong assumptions about the estimand. In contrast to procedures such as training and validation, the Goldenshluger--Lepski method uses all of the data to produce non-adaptive estimators for a range of values of the tuning parameters. An adaptive estimator is selected by performing pairwise comparisons between these non-adaptive estimators. Applying the Goldenshluger--Lepski method is non-trivial as it requires a simultaneous high-probability bound on all of the pairwise comparisons. In the RKHS regression context, we choose our non-adaptive estimators to be clipped least-squares estimators constrained to lie in a ball in an RKHS. Applying the Goldenshluger--Lepski method in this context is made more complicated by the fact that we cannot use the $L^2$ norm for performing the pairwise comparisons as it is unknown. We use the method to address two regression problems. In the first problem the RKHS is fixed, while in the second problem we adapt over a collection of RKHSs.
\end{abstract}

\begin{keyword}
\kwd{Adaptive Estimation}
\kwd{Goldenshluger--Lepski Method}
\kwd{RKHS Regression}
\end{keyword}
\end{frontmatter}

\section{Introduction}

In nonparametric statistics, it is assumed that the estimand belongs to a very large parameter space in order to avoid model misspecification. Such misspecification can lead to large approximation errors and poor estimator performance. However, it is often challenging to produce estimators which are robust against such large parameter spaces. An important tool which allows us to achieve this aim is adaptive estimation. Adaptive estimators behave as if they know the true model from a collection of models, despite being a function of the data. In particular, adaptive estimators can often achieve the same optimal rates of convergence as the best estimators when the true model is known.

In this paper, we study an adaptive estimation procedure called the Goldenshluger--Lepski method in the context of reproducing kernel Hilbert space (RKHS) regression. The Goldenshluger--Lepski method works by performing pairwise comparisons between non-adaptive estimators with a range of values for the tuning parameters. As far as we are aware, this is the first time that this method has been applied in the context of RKHS regression. The Goldenshluger--Lepski method, introduced in the series of papers \cite{goldenshluger2008universal, goldenshluger2009structural, goldenshluger2011bandwidth, goldenshluger2013general}, is an extension of Lepski's method. While Lepski's method focusses on adaptation over a single parameter, the Goldenshluger--Lepski method can be used to perform adaptation over multiple parameters.

The Goldenshluger--Lepski method operates by selecting an estimator which minimises the sum of a proxy for the unknown bias and an inflated variance term. The proxy for the bias is calculated by performing pairwise comparisons between the estimator in question and all estimators which are in some sense less smooth than this estimator. A key challenge in applying the Goldenshluger--Lepski method is proving a high-probability bound on all of these pairwise comparisons simultaneously. This bound is called a majorant.

We now describe the RKHS regression problem studied in this paper. We denote our covariate space by $S$ and we assume that the regression function lies in an interpolation space between $C(S)$, the space of continuous functions equipped with the supremum norm, and an RKHS. Depending on the setting, this RKHS may be fixed or we may perform adaptation over a collection of RKHSs. The non-adaptive estimators we use in this context are clipped versions of least-squares estimators which are constrained to lie in a ball of predefined radius in an RKHS. These estimators are discussed in detail in \cite{page2017ivanov}. Constraining an estimator to lie in a ball of predefined radius is a form of Ivanov regularisation (see \cite{oneto2016tikhonov}).
One advantage of these estimators is that there is a clear way of producing a majorant for them, especially when the RKHS is fixed. This is because we can control the estimator constrained to lie in a ball of radius $r$ by bounding quantities of the form $r Z$ for some random variable $Z$ which does not depend on $r$. 

%It may be possible to use different non-adaptive estimators to address our RKHS regression problem, however this would require the calculation of a majorant for such estimators, which would generally be more difficult than the calculation of the majorant for the Ivanonv-regularised estimators considered in this paper.

When the RKHS is fixed, the only tuning parameter to be selected is the radius of the ball in which the least-squares estimator is constrained to lie. Estimators for which the radius is larger are considered to be less smooth. In order to provide a majorant for the Goldenshluger--Lepski method, we must prove regression results which control these estimators for all radii simultaneously. When we perform adaptation over a collection of RKHSs, we must prove regression results which control the same estimators for all RKHSs and all ball radii in these RKHSs simultaneously. We demonstrate this approach for a collection of RKHSs with Gaussian kernels. Estimators for which both the width parameter of the Gaussian kernel is smaller and the radius of the ball in the RKHS is larger are considered to be less smooth. These results extend those of \cite{page2017ivanov}.

Our main results are Theorems \ref{tInterLepBound} (page \pageref{tInterLepBound}) and \ref{tInterGaussLepBound} (page \pageref{tInterGaussLepBound}). These show that a fixed quantile of the squared $L^2 (P)$ error of a clipped version of the estimator produced by the Goldenshluger--Lepski method is of order $n^{- \beta/(1+\beta)}$. Here, $n$ is the number of data points and $\beta$ parametrises the interpolation space between $C(S)$ and the RKHS containing the regression function. We use $C(S)$ when interpolating so that we have direct control over approximation errors in the $L^2 (P_n)$ norm. Theorem \ref{tInterLepBound} addresses the case in which the RKHS is fixed and Theorem \ref{tInterGaussLepBound} addresses the case in which we perform adaptation over a collection of RKHSs with Gaussian kernels. The order $n^{- \beta/(1+\beta)}$ for the squared $L^2 (P)$ error of the adaptive estimators matches the order of the smallest bounds obtained in \cite{page2017ivanov} for the squared $L^2 (P)$ error of the non-adaptive estimators. In the sense discussed in \cite{page2017ivanov}, this order is the optimal power of $n$ if we make the slightly weaker assumption that the regression function is an element of the interpolation space between $L^2 (P)$ and the RKHS parametrised by $\beta$.
In particular, for a fixed $\beta \in (0, 1)$, \citet{steinwart2009optimal} show that there is an instance of our problem such that the following holds. For all estimators $\hat f$ of $g$, for some $\eps > 0$, we have
\begin{equation*}
\lVert \hat f - g \rVert_{L^2 (P)}^2 \geq C_{\alpha, \eps} n^{- \alpha}
\end{equation*}
with probability at least $\eps$ for all $n \geq 1$, for some constant $C_{\alpha, \eps} > 0$, for all $\alpha > \beta/(1 + \beta)$. This implies that for all estimators $\hat f$ of $g$, we have
\begin{equation*}
\expect \left (\lVert \hat f - g \rVert_{L^2 (P)}^2 \right ) \geq C_{\alpha, \eps} \eps n^{- \alpha}
\end{equation*}
for all $n \geq 1$, for all $\alpha > \beta/(1 + \beta)$. In this sense, our expectation bound in this setting is optimal because it attains the order $n^{- \beta/(1 + \beta)}$, the smallest possible power of $n$. 

Under further assumptions on the effective dimension of the RKHS faster rates of convergence can be obtained. For instance, in \citet{steinwart2009optimal} optimal rates of convergence under assumptions on the effective dimension are attained and our rates correspond to their worst-case rates. Also see \cite{LU19,BLAN19} for results on estimating the effective dimension from data.

\subsection{Literature Review}

Lepski's method was introduced in the series of papers \cite{lepski1991asymptotically, lepskii1991problem, lepskii1993asymptotically} as a method for adaptation over a single parameter. It has since been studied in, for example, \cite{birge2001alternative} and \cite{gine2015mathematical}. Lepski's method selects the smoothest non-adaptive estimator from a collection, subject to a bound on a series of pairwise comparisons involving all estimators at most as smooth as the resulting estimator. The method can only adapt to one parameter because of the need for an ordering of the collection of non-adaptive estimators.

Lepski's method has been applied to RKHS regression under the name of the balancing principle. However, as far as we are aware, Lepski's method has not been used to target the true regression function, but instead an RKHS element which approximates the true regression function. In \cite{de2010adaptive}, the authors note the difficulty in using Lespki's method to control the squared $L^2 (P)$ error of an adaptive version of a support vector machine (SVM). This difficulty arises because Lepski's method generally requires the norm we are interested in controlling to be known in order to perform the pairwise comparisons. However, $P$ is unknown in this situation.
The authors of \cite{de2010adaptive} get around the problem that $P$ is unknown as follows. Lepski's method is used to control the known squared $L^2 (P_n)$ error and squared RKHS error of two different adaptive SVMs. The results of these procedures are combined to produce an adaptive SVM whose squared $L^2 (P)$ error is bounded. The above alteration is also noted in \cite{lu2018balancing}. Furthermore, the authors show that it is possible to greatly reduce the number of pairwise comparisons which must be performed to produce an adaptive estimator. This is done by only comparing each estimator to the estimator which is next less smooth.

Lepski's method is common in the inverse problems literature and is there also known as the balancing principle. In this context it is not assumed that the target function lies in a particular RKHS but within a scale of Hilbert spaces. 
Scales of Hilbert spaces have been developed in \cite{KR66} and  have been introduced to the inverse problems literature in \cite{MAIR96} and \cite{MAT01}. These scales of Hilbert spaces are closely linked to interpolation spaces \cite{KR66}.  
Adaptation over these scales is typically achieved by 
combining Tikhonov regularisation with Lepski's method.  This approach achieves in various settings close to optimal rates of convergence (e.g. \cite{GOLD03, PRI19}).

The Goldenshluger--Lepski method extends Lepski's method in order to perform adaptation over multiple parameters. The method is introduced in the series of papers \cite{goldenshluger2008universal, goldenshluger2009structural, goldenshluger2011bandwidth, goldenshluger2013general}. The first two papers concentrate on function estimation in the presence of white noise. The first paper considers the problem of pointwise estimation, while the second paper examines estimation in the $L^p$ norm for $p \in [1, \infty]$. The third paper produces adaptive bandwidth estimators for kernel density estimation and the fourth paper considers general methodology for selecting a linear estimator from a collection.

Another popular approach to model selection is to use a  training and validation set. The training set is used to produce a collection of non-adaptive estimators for a range of different values for the tuning parameters and the validation set is used to select the best estimator from this collection. This selection is performed by calculating a proxy for the cost function that we wish to minimise. An example of using training and validation to perform adaptation over a Gaussian kernel parameter for an SVM can be found in \cite{eberts2013optimal}. The procedure produces an adaptive estimator of a bounded regression function from a range of Sobolev spaces. This estimator is analysed using union bounding, as opposed to the chaining techniques used to analyse the Goldenshluger--Lepski method in this paper. The training and validation approach is studied in a more general context in \cite{yao2010} using inverse problem techniques.  One important advantage of the Goldenshluger--Lepski method in comparison to training and validation is that it uses all of the data to calculate the non-adaptive estimators. This is because it does not require data for calculating a proxy cost function. However, the Goldenshluger--Lepski method does require us to calculate a majorant, as discussed above, which is often a challenging task.

There are various other approaches to gain adaptive estimators. For instance, in \cite{BAR02} regression under random design is studied. The estimators are chosen adaptively from a family of finite dimensional subspaces of $L^2(\nu)$, where $\nu$ is a known measure on the covariate space. The estimation approach selects a member of the family of subspaces by minimising a linear combination of the best least-squares error that can be achieved within the class and a penalty term which penalises model complexity. In contrast to Lepski's method the model class is selected based solely on this value  and no pairwise comparison between model classes is performed.
Model selection based on penalised least-squares values is a well-studied area, in particular, in the Gaussian noise case. To mention but a few references, in  \cite{BIR01} a versatile framework for model selection under a penalised least-squares criterion is developed and in \cite{BIR07} the role of the penalty term in this framework is analysed in depth. Closely related are earlier works  \cite{BIR97,BAR99} where penalty terms are used in conjunction with sieves to produce adaptive estimators.

We perform in this paper adaptation over a family of kernel functions. This approach touches on the topic of multiple kernel learning (MKL) where a family of kernels is given and a regressor is chosen from the corresponding family of RKHSs.  Early MKL papers like \cite{CHAP02,BACH04} focus on the optimisation aspect of the problem, i.e. how to efficiently find a suitable kernel within a large family of kernels to gain small empirical errors.  In \cite{GON11} a broad overview of MKL techniques up to the year 2011 is given. A more recent overview is contained in \cite{PER17}. Much of the MKL literature focuses solely on the optimisation problem and not on controlling the risk of estimators. To the best of our knowledge,  the risk has not been studied in the MKL literature under our assumption that the unknown regression function can lie outside of all the involved RKHSs.

\section{Problem Definition}

We give a formal definition of the RKHS regression problem. First, recall that an RKHS $H$ on $S$ is a Hilbert space of real-valued functions on $S$ such that, for all $x \in S$, there is some $k_x \in H$ such that $h(x) = \langle h, k_x \rangle_H$ for all $h \in H$. The function $k(x_1, x_2) = \langle k_{x_1}, k_{x_2} \rangle_H$ for $x_1, x_2 \in S$ is known as the kernel and is symmetric and positive-definite.

For a topological space $T$, let $\bor(T)$ be its Borel $\sigma$-algebra. Let $(S, \mathcal{S})$ be a measurable space and $(X_i, Y_i)$ for $1 \leq i \leq n$ be \iid $(S \times \real, \mathcal{S} \otimes \bor(\real))$-valued random variables on the probability space $(\Omega, \mathcal{F}, \prob)$. We assume $X_i \sim P$ and $\expect(Y_i^2) < \infty$, where $\expect$ denotes integration with respect to $\prob$. We have $\expect(Y_i \vert X_i) = g(X_i)$ almost surely for some function $g$ which is measurable on $(S, \mathcal{S})$ (Section A3.2 of \cite{williams1991probability}). Our goal is to estimate $g$ from the data $(X_i, Y_i)$, $1 \leq i \leq n$. Since $\expect(Y_i^2) < \infty$, it follows that $g \in L^2 (P)$ by Jensen's inequality. We consider the quadratic loss with the corresponding risk function for an estimate $\hat h$ given by
\[
\expect(\hat h(X) - Y)^2,
\]
where the pair $(X,Y)$ has the same distribution as $(X_1,Y_1)$ and is independent of $(X_i,Y_i), 1 \leq i \leq n$. Furthermore, $\hat h$ is assumed measurable with respect to $(X_i,Y_i), 1 \leq i \leq n$, and to attain values in an RKHS or in a collection of RKHSs.

\newpage

We assume throughout that
\begin{enumerate}[label={($g$1)}]
\item \hfil $\lVert g \rVert_\infty \leq C$ for $C > 0$. \label{g1}
\end{enumerate}

\vspace{0.5\baselineskip}

We also need to make an assumption on the behaviour of the errors of the response variables $Y_i$ for $1 \leq i \leq n$. Let $U$ and $V$ be random variables on $(\Omega, \mathcal{F}, \prob)$. We say $U$ is $\sigma^2$-subgaussian if
\begin{equation*}
\expect(\exp(t U)) \leq \exp(\sigma^2 t^2 / 2)
\end{equation*}
for all $t \in \real$. We say $U$ is $\sigma^2$-subgaussian given $V$ if
\begin{equation*}
\expect(\exp(t U) \vert V) \leq \exp(\sigma^2 t^2 / 2)
\end{equation*}
almost surely for all $t \in \real$. We assume
\begin{enumerate}[label={($Y$)}]
\item \hfil $Y_i - g(X_i)$ is $\sigma^2$-subgaussian given $X_i$ for $1 \leq i \leq n$. \label{Y}
\end{enumerate}

\section{Regression for a Fixed RKHS}

We continue by providing simultaneous bounds on our collection of non-adaptive estimators for a fixed RKHS. Our results in this section depend on how well the regression function $g$ can be approximated by elements of an RKHS $H$ with kernel $k$. We make the following assumptions.
\begin{enumerate}[label={($H$)}]
\item {The RKHS $H$ with kernel $k$ has the following properties:
\begin{itemize}
\item The RKHS $H$ is separable.
\item The kernel $k$ is bounded.
\item The kernel $k$ is a measurable function on $(S \times S, \mathcal{S} \otimes \mathcal{S})$.
\end{itemize}} \label{H}
\end{enumerate}
We define
\begin{equation*}
\lVert k \rVert_{\diag} = \sup_{x \in S} k(x, x) < \infty.
\end{equation*}
We can guarantee that $H$ is separable by, for example, assuming that $k$ is continuous and $S$ is a separable topological space (Lemma 4.33 of \cite{steinwart2008support}). The fact that $H$ has a kernel $k$ which is measurable on $(S \times S, \mathcal{S} \otimes \mathcal{S})$ guarantees that all functions in $H$ are measurable on $(S, \mathcal{S})$ (Lemma 4.24 of \cite{steinwart2008support}).

Let $B_H$ be the closed unit ball of $H$ and $r > 0$. We define the estimator
\begin{equation*}
\hat h_r = \argmin_{f \in r B_H} \frac{1}{n} \sum_{i = 1}^n (f(X_i) - Y_i)^2
\end{equation*}
of the regression function $g$. We make this definition unique by demanding that $\hat h_r \in \spn \{ k_{X_i} : 1 \leq i \leq n \}$ (see Lemma 3 of \cite{page2017ivanov}). We also define $\hat h_0 = 0$. 
Lemma \ref{lEst} on page \ref{lEst}  states that this estimator is well defined.

Since we assume \ref{g1}, that $g$ is bounded in $[- C, C]$, we can make $\hat h_r$ closer to $g$ by constraining it to lie in the same interval. As in \cite{page2017ivanov}, we define the projection $V : \real \to [-C, C]$ by
\begin{equation*}
V(t) =\left \{
\begin{array}{lll}
- C & \text{if} &t < - C \\
t & \text{if} &\lvert t \rvert \leq C \\
C & \text{if} &t > C
\end{array}
\right .
\end{equation*}
for $t \in \real$. For a function $f:S\rightarrow \mathbb{R}$ we write $V f$ for the composition of $V$ and $f$.

We now prove a series of result which allow us to control $\hat h_r$ for $r \geq 0$ simultaneously, extending the results of \cite{page2017ivanov} while using similar proof techniques. This is crucial in order to apply the Goldenshluger--Lepski method to these estimators. The results assign probabilities to events which occur for all $r \geq 0$ and all $h_r \in r B_H$. These events are measurable due to the separability of $[0, \infty)$ and $r B_H$, as well as the continuity in $r$ of the quantities in question, including $\hat h_r$ by Lemma \ref{lEst}. By Lemma 2 of \cite{page2017ivanov}, we have
\begin{equation*}
\lVert \hat h_r -  h_r \rVert_{L^2 (P_n)}^2 \leq \frac{4}{n} \sum_{i = 1}^n (Y_i - g(X_i)) (\hat h_r (X_i) - h_r (X_i)) + 4 \lVert h_r - g \rVert_{L^2 (P_n)}^2
\end{equation*}
for all $r > 0$ and all $h_r \in r B_H$. We can remove $\hat h_r$ in the first term on the right-hand side by taking a supremum over $r B_H$. After applying the reproducing kernel property and the Cauchy--Schwarz inequality, we obtain a quadratic form of subgaussians which can be controlled using Lemma 36 of \cite{page2017ivanov}. See Lemma 22 in \cite{page2018supplement} on page 1 for details.

It is useful to be able to transfer a bound on the squared $L^2 (P_n)$ error of an estimator, including the result above, to a bound on the squared $L^2 (P)$ error of the estimator. By using Talagrand's inequality, we can obtain a high-probability bound on
\begin{equation*}
\sup_{r > 0} \sup_{f_1, f_2 \in r B_H} \frac{1}{r} \left \lvert \lVert V f_1 - V f_2 \rVert_{L^2 (P_n)}^2 - \lVert V f_1 - V f_2 \rVert_{L^2 (P)}^2 \right \rvert
\end{equation*}
by proving an expectation bound on the same quantity. By using symmetrisation (Lemma 2.3.1 of \cite{van1996weak}) and the contraction principle for Rademacher processes (Theorem 3.2.1 of \cite{gine2015mathematical}), we again obtain a quadratic form of subgaussians, which in this case are Rademacher random variables. The result is stated in Lemma 24 in \cite{page2018supplement} on page 3.

To capture how well $g$ can be approximated by elements of $H$, we define
\begin{equation*}
I_\infty (g, r) = \inf \left \{ \lVert h_r - g \rVert_\infty^2 : h_r \in r B_H \right \}
\end{equation*}
for $r \geq 0$. We use this measure of approximation as it is compatible with the use of the bound
\begin{equation*}
\lVert h_r - g \rVert_{L^2 (P_n)}^2 \leq \lVert h_r - g \rVert_\infty^2
\end{equation*}
in the proof of Lemma 22. It is easy to show that $I_\infty (g, r)$ is continuous in $r$ under Assumption \ref{H}.

We obtain a bound on the squared $L^2 (P)$ error of $V \hat h_r$ by combining Lemmas 22 and 24.

\begin{theorem} \label{tBound}
Assume \ref{g1}, \ref{Y} and \ref{H}. Let $t \geq 1$ and recall the definitions of $A_{1, t}$ and $A_{2, t}$ from Lemmas 22 and 24. On the set $A_{1, t} \cap A_{2, t} \in \mathcal{F}$, for which $\prob(A_{1, t} \cap A_{2, t}) \geq 1 - 2 e^{-t}$, we have
\begin{equation*}
\lVert V \hat h_r - g \rVert_{L^2 (P)}^2 \leq \frac{2 \lVert k \rVert_{\diag}^{1/2} (97 C + 20 \sigma) r t^{1/2}}{n^{1/2}} + \frac{16 \lVert k \rVert_{\diag}^{1/2} C r t}{3 n} + 10 I_\infty (g, r)
\end{equation*}
simultaneously for all $r \geq 0$.
\end{theorem}

\subsection{The Goldenshluger--Lepski Method for a Fixed RKHS}

We now produce bounds on our adaptive estimator for a fixed RKHS. Lemma \ref{lComp} (p. \pageref{lComp}), which is a simple consequence of Lemma 22 can be used to define the majorant of the non-adaptive estimators. This motivates the following definition of the adaptive estimator used in the Goldenshluger--Lepski method.

Let $R \subs [0, \infty)$ be closed and non-empty. The Goldenshluger--Lepski method defines an adaptive estimator using
\begin{equation}
\hat r = \argmin_{r \in R} \left ( \sup_{s \in R, s \geq r} \left ( \lVert \hat h_r - \hat h_s \rVert_{L^2 (P_n)}^2 - \frac{\tau (r + s)}{n^{1/2}} \right ) + \frac{2 (1 + \nu) \tau r}{n^{1/2}} \right ) \label{eEst}
\end{equation}
for tuning parameters $\tau, \nu > 0$. The supremum of pairwise comparisons can be viewed as a proxy for the unknown bias, while the other term is an inflated variance term. Note that the supremum is at least the value at $r$, so
\begin{equation} \label{eCrit}
\sup_{s \in R, s \geq r} \left ( \lVert \hat h_r - \hat h_s \rVert_{L^2 (P_n)}^2 - \frac{\tau (r + s)}{n^{1/2}} \right ) + \frac{2 (1 + \nu) \tau r}{n^{1/2}} \geq \frac{2 \nu \tau r}{n^{1/2}}.
\end{equation}
The role of the tuning parameter $\nu$ is simply to control this bound. In fact, the parameter $\nu$ only affects the constants and not the rate of convergence in Theorem \ref{tInterLepBound} and Theorem \ref{tInterGaussLepBound}. 
The parameter $\tau$ controls the probability with which our bound on the squared $L^2 (P)$ error of $V \hat h_{\hat r}$ holds.  Lemma \ref{lEstWD} on page \pageref{lEstWD} shows that the estimator $\hat r$ is well-defined by Equation \eqref{eEst}.

It may be that $\hat r$ is not a random variable on $(\Omega, \mathcal{F})$ in some cases, but we assume
\begin{enumerate}[label={($\hat r$)}]
\item \hfil $\hat r$ is a well-defined random variable on $(\Omega, \mathcal{F})$ \label{hr}
\end{enumerate}
throughout. Later, we assume that $R$ is finite, in which case $\hat r$ is certainly a random variable on $(\Omega, \mathcal{F})$. If $\hat r$ is a random variable on $(\Omega, \mathcal{F})$, then $\hat h_{\hat r}$ is a $(H, \bor(H))$-valued measurable function on $(\Omega, \mathcal{F})$ by Lemma \ref{lEst}.

By Lemma \ref{lComp}, the supremum in the definition of $\hat r$ is at most $40 I_\infty (g, r)$ for an appropriate value of $t$. The definition of $\hat r$ then gives us control over the squared $L^2 (P_n)$ norm of $\hat h_{\hat r} - \hat h_r$ when $\hat r \leq r$. When $\hat r \geq r$, we can control the squared $L^2 (P_n)$ norm of $\hat h_{\hat r} - \hat h_r$ using Lemma \ref{lComp}. However, we must control a term of order $\hat r / n^{1/2}$ using \eqref{eCrit} and the definition of $\hat r$. In both cases, this gives a bound on the squared $L^2 (P_n)$ norm of $V \hat h_{\hat r} - V \hat h_r$. Extra terms appear when moving to a bound on the squared $L^2 (P)$ norm of $V \hat h_{\hat r} - V \hat h_r$ using Lemma 24. However, these terms are very similar to the inflated variance term, and can be controlled in the same way. Applying
\begin{equation*}
\lVert V \hat h_{\hat r} - g \rVert_{L^2 (P)}^2 \leq 2 \lVert V \hat h_{\hat r} - V \hat h_r \rVert_{L^2 (P)}^2 + 2 \lVert V \hat h_r - g \rVert_{L^2 (P)}^2
\end{equation*}
leads to Theorem \ref{tLep} (p. \pageref{tLep}). Combining it with Theorem \ref{tBound}
and using the events 
\begin{align*}
A_{1,t} &= \Bigl\{\lVert \hat h_r - h_r \rVert_{L^2 (P_n)}^2 \leq \frac{20 \lVert k \rVert_{\diag}^{1/2} \sigma r t^{1/2}}{n^{1/2}} + 4 \lVert h_r - g \rVert_\infty^2, \text{ for all } r\geq 0\Bigr\}, \\
A_{2,t} &= \Bigl\{\sup_{f_1, f_2 \in r B_H} \left \lvert \lVert V f_1 - V f_2 \rVert_{L^2 (P_n)}^2 - \lVert V f_1 - V f_2 \rVert_{L^2 (P)}^2 \right \rvert \\
&\qquad\qquad\qquad\qquad\enspace \leq \frac{97 \lVert k \rVert_{\diag}^{1/2} C r t^{1/2}}{n^{1/2}} + \frac{8 \lVert k \rVert_{\diag}^{1/2} C r t}{3 n}, \text{ for all } r\geq 0\Bigr\},
\end{align*}
gives us the following main result.

\begin{theorem} \label{tLepBound}
Assume \ref{g1}, \ref{Y}, \ref{H} and \ref{hr}. Let $\tau \geq 80 \lVert k \rVert_{\diag}^{1/2} \sigma$, $\nu > 0$ and
\begin{equation*}
t = \left ( \frac{\tau}{80 \lVert k \rVert_{\diag}^{1/2} \sigma} \right )^2 \geq 1.
\end{equation*}
%Recall the definitions of $A_{1, t}$ and $A_{2, t}$ from Lemmas \ref{lBias} and \ref{lSwapNorm}. 
On the set $A_{1, t} \cap A_{2, t} \in \mathcal{F}$, for which $\prob(A_{1, t} \cap A_{2, t}) \geq 1 - 2 e^{-t}$, we have
\begin{equation*}
\lVert V \hat h_{\hat r} - g \rVert_{L^2 (P)}^2 \leq \inf_{r \in R} \left ( (1 + D_1 \tau n^{-1/2}) (D_2 \tau r n^{-1/2} + D_3 I_\infty (g, r)) \right )
\end{equation*}
for constants $D_1, D_2, D_3 > 0$ not depending on $\tau$, $r$ or $n$.
\end{theorem}

We can obtain rates of convergence for our estimator $V \hat h_{\hat r}$ if we make an assumption about how well $g$ can be approximated by elements of $H$. Such assumptions are typically stated in terms of interpolation spaces between a Banach space $(Z, \lVert \cdot \rVert_Z)$ and a dense subspace $(V, \lVert \cdot \rVert_V)$ (see \cite{bergh2012interpolation}). The $K$-functional of $(Z, V)$ is
\begin{equation*}
K(z, t) = \inf_{v \in V} (\lVert z - v \rVert_Z + t \rVert v \lVert_V)
\end{equation*}
for $z \in Z$ and $t > 0$. We define
\begin{equation*}
\lVert z \rVert_{\beta, q} = \left ( \int_0^\infty ( t^{-\beta} K(z, t) )^q t^{-1} dt \right )^{1/q} \text{ and } \lVert z \rVert_{\beta, \infty} = \sup_{t > 0} (t^{-\beta} K(z, t))
\end{equation*}
for $z \in Z$, $\beta \in (0, 1)$ and $1 \leq q < \infty$. We then define the interpolation space $[Z, V]_{\beta, q}$ to be the set of $z \in Z$ such that $\lVert z \rVert_{\beta, q} < \infty$. The size of $[Z, V]_{\beta, q}$ decreases as $\beta$ increases. The following result is Lemma 1 of \cite{page2017ivanov}, which is essentially Theorem 3.1 of \cite{smale2003estimating}.

\begin{lemma} \label{lApproxInter}
Let $(Z, \lVert \cdot \rVert_Z)$ be a Banach space, $(V, \lVert \cdot \rVert_V)$ be a dense subspace of $Z$ and $z \in [Z, V]_{\beta, \infty}$. We have
\begin{equation*}
\inf \{ \lVert v - z \rVert_Z : v \in V, \lVert v \rVert_V \leq r \} \leq \frac{\lVert z \rVert_{\beta, \infty}^{1/(1-\beta)}}{r^{\beta/(1-\beta)}}.
\end{equation*}
\end{lemma}
Results like this have a long history in inverse problem theory \cite{BAU87,HOF05}.
From the above, when $H$ is dense in $C(S)$, we can define the interpolation spaces $[C(S), H]_{\beta, q}$. We set $q = \infty$ and work with the largest space of functions for a fixed $\beta \in (0, 1)$. We are then able to apply the approximation result in Lemma \ref{lApproxInter}.

Let us assume
\begin{enumerate}[label={($g$2)}]
\item \hfil $g \in [C(S), H]_{\beta, \infty}$ with norm at most $B$ for $\beta \in (0, 1)$ and $B > 0$. \label{g2}
\end{enumerate}

The assumption \ref{g2}, together with Lemma \ref{lApproxInter}, give
\begin{equation}
I_\infty (g, r) \leq \frac{B^{2/(1-\beta)}}{r^{ 2 \beta/(1-\beta)}} \label{eApproxInter}
\end{equation}
for $r > 0$. In order for us to apply Theorem \ref{tLepBound} to this setting, we need to make an assumption on $R$. We assume either
\begin{enumerate}[label={($R$1)}]
\item \hfil $R = [0, \infty)$ \label{R1}
\end{enumerate}
or
\begin{enumerate}[label={($R$2)}]
\item \hfil $R = \{ b i : 0 \leq i \leq I - 1 \} \cup \{ a n^{1/2} \}$ and $\rho = a n^{1/2}$ for $a, b > 0$ and $I = \lceil a n^{1/2} / b \rceil$. \label{R2}
\end{enumerate}
The assumption \ref{R1} is mainly of theoretical interest and would make it difficult to calculate $\hat r$ in practice. The estimator $\hat r$ can be computed under the assumption \ref{R2}, since in this case $R$ is finite. We obtain a high-probability bound on a fixed quantile of the squared $L^2 (P)$ error of $V \hat h_{\hat r}$ of order $t^{1/2} n^{-\beta/(1+\beta)}$ with probability at least $1 - e^{- t}$ when $\tau$ is an appropriate multiple of $t^{1/2}$.

\begin{corollary} \label{tInterLepBound}
Assume \ref{g1}, \ref{g2}, \ref{Y} and \ref{H}. Let $\tau \geq 80 \lVert k \rVert_{\diag}^{1/2} \sigma$, $\nu > 0$ and
\begin{equation*}
t = \left ( \frac{\tau}{80 \lVert k \rVert_{\diag}^{1/2} \sigma} \right )^2 \geq 1.
\end{equation*}
Also assume \ref{R1} and \ref{hr}, or \ref{R2}.  On the set $A_{1, t} \cap A_{2, t} \in \mathcal{F}$, for which $\prob(A_{1, t} \cap A_{2, t}) \geq 1 - 2 e^{-t}$, we have
\begin{equation*}
\lVert V \hat h_{\hat r} -  g \rVert_{L^2 (P)}^2 \leq D_1 \tau n^{-\beta/(1+\beta)} + D_2 \tau^2 n^{-(1+3\beta)/(2(1+\beta))}
\end{equation*}
for constants $D_1, D_2 > 0$ not depending on $n$ or $\tau$.
\end{corollary}
From the high probability bound we obtain a bound on the expected squared $L^2(P)$ error.
\begin{corollary} \label{InterLepExpBound}
Assume \ref{g1}, \ref{g2}, \ref{Y}, \ref{H} and either \ref{R1} and \ref{hr}, or \ref{R2}. We have 
\[
\expect(\lVert V \hat h_{\hat r} -  g \rVert_{L^2 (P)}^2) \leq D n^{-\beta/(1+\beta)}
\]
for constant $D$ not depending on $n$.
\end{corollary}

\section{Regression for a Collection of RKHSs}

In this section, we again provide simultaneous bounds on our collection of non-adaptive estimators. Our results still depend on how well the regression function $g$ can be approximated by elements of an RKHS. However, this RKHS now comes from a collection instead of being fixed. Let $\mathcal{K}$ be a set of kernels on $S \times S$. We make the following assumptions.
\begin{enumerate}[label={($\mathcal{K}1$)}]
\item {The covariate set $S$ and the set of kernels $\mathcal{K}$ have the following properties:
\begin{itemize}
\item The covariate set $S$ is a separable topological space.
\item The set of kernels $(\mathcal{K}, \lVert \cdot \rVert_\infty)$ is separable.
\item The kernel $k$ is bounded for all $k \in \mathcal{K}$.
\item The kernel $k$ is continuous for all $k \in \mathcal{K}$.
\end{itemize}} \label{K1}
\end{enumerate}
Since $(\mathcal{K}, \lVert \cdot \rVert_\infty)$ is a separable set of kernels, we have that $\mathcal{K}$ has a countable dense subset $\mathcal{K}_0$. For all $\eps > 0$ and all $k \in \mathcal{K}$, there exists $k_0 \in \mathcal{K}_0$ such that
\begin{equation*}
\lVert k_0 - k \rVert_\infty = \sup_{x_1, x_2 \in S} \lvert k_0 (x_1, x_2) - k(x_1, x_2) \rvert < \eps.
\end{equation*}
Let $H_k$ be the RKHS with kernel $k$ for $k \in \mathcal{K}$. Since $k$ is continuous and $S$ is a separable topological space, we have that $H_k$ is separable by Lemma 4.33 of \cite{steinwart2008support}. Hence, the assumption \ref{H} holds for $H_k$. We use the notation $\lVert \cdot \rVert_k$ and $\langle \cdot, \cdot \rangle_k$ for the norm and inner product of $H_k$.

Let $B_k$ be the closed unit ball of $H_k$ for $k \in \mathcal{K}$ and $r > 0$. We define the estimator
\begin{equation*}
\hat h_{k, r} = \argmin_{f \in r B_k} \frac{1}{n} \sum_{i = 1}^n (f(X_i) - Y_i)^2
\end{equation*}
of the regression function $g$. We make this definition unique by demanding that $\hat h_{k, r} \in \spn \{ k_{X_i} : 1 \leq i \leq n \}$ (see Lemma 3 of \cite{page2017ivanov}). We also define $\hat h_{k, 0} = 0$. Since we assume \ref{g1}, that $g$ is bounded in $[- C, C]$, we can make $\hat h_{k, r}$ closer to $g$ by clipping it to obtain $V \hat h_{k, r}$. Lemma 26 (\cite{page2018supplement}, p. 6) shows that $\hat h_{k,r}$ is a well-defined estimator.

Let
\begin{equation*}
\mathcal{L} = \{ k / \lVert k \rVert_{\diag} : k \in \mathcal{K} \} \cup \{0\}
\end{equation*}
and
\begin{equation*}
D = \sup_{f_1, f_2 \in \mathcal{L}} \lVert f_1 - f_2 \rVert_\infty \leq 2.
\end{equation*}
We include 0 in the definition of $\mathcal{L}$ so that, when analysing stochastic processes over $\mathcal{L}$ using chaining, we can start all chains at 0. Note that $(\mathcal{L}, \lVert \cdot \rVert_\infty)$ is separable since $\mathcal{L} \setminus \{0\}$ is the image of a continuous function on $(\mathcal{K}, \lVert \cdot \rVert_\infty)$, which is itself separable. Let $N(a, M, d)$ be the minimum size of an $a > 0$ cover of a metric space $(M, d)$. The following quantity is key to quantifying the complexity of our collection of kernels, 
\begin{equation*}
J = \left ( 162 \int_{0}^{D/2} \log(2 N(a, \mathcal{L}, \lVert \cdot \rVert_\infty)) da + 1 \right )^{1/2}.
\end{equation*}
Using $J$ we upper bound  the empirical $L^2$-distance between the estimator $\hat h_{k,r}$ and $h_{k,r}$ in Lemma 29 (\cite{page2018supplement}, p. 10). 
Instead of one quadratic form of subgaussians as in Lemma 31, we obtain a supremum over $\mathcal{K}$ of quadratic forms of subgaussians. This can be controlled by chaining using Lemma 28 of \cite{page2018supplement}.

It is again useful to be able to transfer a bound on the squared $L^2 (P_n)$ error of an estimator to a bound on the squared $L^2 (P)$ error of the estimator. Lemma 31 (\cite{page2018supplement}, p. 12) allows us to do such a transfer.
The result is proved using the same method as Lemma 24, 
although we  obtain a supremum of quadratic forms of subgaussians which are controlled using chaining. The event in the result is measurable by Lemma 30 of \cite{page2018supplement}.

To capture how well $g$ can be approximated by elements of $H_k$, we define
\begin{equation*}
I_\infty (g, k, r) = \inf \left \{ \lVert h_{k, r} - g \rVert_\infty^2 : h_{k, r} \in r B_k \right \}
\end{equation*}
for $k \in \mathcal{K}$ and $r \geq 0$. We obtain a bound on the squared $L^2 (P)$ error of $V \hat h_{k, r}$ by combining Lemmas 29 and 31.

\begin{theorem} \label{tVaryBound}
Assume \ref{g1}, \ref{Y} and \ref{K1}. Let $t \geq 1$ and recall the definitions of $A_{3, t}$ and $A_{4, t}$ from Lemmas 29 and 31. On the set $A_{3, t} \cap A_{4, t} \in \mathcal{F}$, for which $\prob(A_{3, t} \cap A_{4, t}) \geq 1 - 2 e^{-t}$, we have
\begin{equation*}
\lVert V \hat h_{k, r} - g \rVert_{L^2 (P)}^2 \leq \frac{2 J \lVert k \rVert_{\diag}^{1/2} (151 C + 21 \sigma) r t^{1/2}}{n^{1/2}} + \frac{16 \lVert k \rVert_{\diag}^{1/2} C r t}{3 n} + 10 I_\infty (g, k, r)
\end{equation*}
simultaneously for all $k \in \mathcal{K}$ and all $r \geq 0$.
\end{theorem}

\subsection{The Goldenshluger--Lepski Method for a Collection of RKHSs with Gaussian Kernels}

We now apply the Goldenshluger--Lepski method again in the context of RKHS regression. However, we now produce an estimator which adapts over a collection of RKHSs with Gaussian kernels. We make the following assumptions on $S$ and $\mathcal{K}$.
\begin{enumerate}[label={($\mathcal{K}2$)}]
\item {The covariate set $S$ and the set of kernels $\mathcal{K}$ have the following properties:
\begin{itemize}
\item The covariate set $S \subs \real^d$ for $d \geq 1$.
\item {The set of kernels
\begin{equation*}
\mathcal{K} = \left \{ k_\gamma (x_1, x_2) = \gamma^{-d} \exp \left (- \lVert x_1 - x_2 \rVert_2^2 / \gamma^2 \right ) : \gamma \in \Gamma \text{ and } x_1, x_2 \in S \right \}
\end{equation*}
for $\Gamma \subs [u, v]$ non-empty for $v \geq u > 0$.}
\end{itemize}} \label{K2}
\end{enumerate}
Recalling the definitions from the previous section, we have
\begin{equation*}
\mathcal{L} = \left \{ f_\gamma (x_1, x_2) = \exp \left (- \lVert x_1 - x_2 \rVert_2^2 / \gamma^2 \right ) : \gamma \in \Gamma \text{ and } x_1, x_2 \in S \right \} \cup \{0\}.
\end{equation*}
The assumption \ref{K2} implies the assumption \ref{K1}. This is because Lemma 33 of \cite{page2018supplement} shows that $(\mathcal{L}, \lVert \cdot \rVert_\infty)$, and hence $(\mathcal{K}, \lVert \cdot \rVert_\infty)$, is separable. We change notation slightly. Let $H_\gamma$ be the RKHS with kernel $k_\gamma$ for $\gamma \in \Gamma$, let $\lVert \cdot \rVert_\gamma$ and $\langle \cdot, \cdot \rangle_\gamma$ be the norm and inner product of $H_\gamma$, and let $B_\gamma$ be the closed unit ball of $H_\gamma$. Furthermore, we write $\hat h_{\gamma, r}$ in place of $\hat h_{k_\gamma, r}$ and $I_\infty (g, \gamma, r)$ in place of $I_\infty (g, k_\gamma, r)$.

The scaling of the kernels is selected so that the following lemma holds. The result is immediate from Proposition 4.46 of \cite{steinwart2008support} and the way that the norm of an RKHS scales with its kernel (Theorem 4.21 of \cite{steinwart2008support}).

\begin{lemma} \label{lNest}
Assume \ref{K2}. Let $\gamma, \eta \in \Gamma$ with $\gamma \geq \eta$. We have $B_\gamma \subs B_\eta$.
\end{lemma}

By Lemma 33 of \cite{page2018supplement}, the function $F : \Gamma \to \mathcal{L} \setminus \{0\}$ by $F(\gamma) = f_\gamma$ is continuous. Hence, the function $G : \Gamma \to \mathcal{K}$ by $G(\gamma) = k_\gamma$ is continuous. Lemma 34 (\cite{page2018supplement}, p. 17), which follows directly from Lemma 26,  now states that the estimator is well-defined.

Recall the definition of $J$ from the previous section. Lemma 36 (\cite{page2018supplement}, p. 17) provides us with a bound on $J$. Furthermore,
Lemma 15 on page 23 can be used to define the majorant of the non-adaptive estimators and is a simple consequence of Lemma 29. This motivates the definition of the adaptive estimator used in the Goldenshluger--Lepski method.

Let $R \subs [0, \infty)$ be non-empty. The Goldenshluger--Lepski method creates an adaptive estimator by defining $(\hat \gamma, \hat r)$ to be the minimiser of
\begin{equation}
\sup_{\eta \in \Gamma, \eta \leq \gamma} \sup_{s \in R, s \geq r} \left ( \lVert \hat h_{\gamma, r} - \hat h_{\eta, s} \rVert_{L_2 (P_n)}^2 - \frac{\tau (\gamma^{-d/2} r + \eta^{-d/2} s)}{n^{1/2}} \right ) + \frac{2 (1 + \nu) \tau \gamma^{-d/2} r}{n^{1/2}} \label{eGaussEst}
\end{equation}
over $(\gamma, r) \in \Gamma \times R$ for tuning parameters $\tau, \nu > 0$. Again, the supremum of pairwise comparisons can be viewed as a proxy for the unknown bias, while the other term is an inflated variance term. Note that the supremum is at least the value at $(\gamma, r)$, which means the above is at least
\begin{equation} \label{eGaussCrit}
\frac{2 \nu \tau \gamma^{-d/2} r}{n^{1/2}}.
\end{equation}
Again, the role of the tuning parameter $\nu$ is simply to control this bound. The parameter $\tau$ controls the probability with which our bound on the squared $L^2 (P)$ error of $V \hat h_{\hat \gamma, \hat r}$ holds. It may be that $\hat \gamma$ is not a well-defined random variable on $(\Omega, \mathcal{F})$ in some cases, but we assume
\begin{enumerate}[label={($\hat \gamma$)}]
\item \hfil $\hat \gamma$ is a well-defined random variable on $(\Omega, \mathcal{F})$ \label{hgamma}
\end{enumerate}
throughout. Later, we assume that $R$ and $\Gamma$ are finite, in which case $\hat \gamma$ and $\hat r$ are certainly well-defined random variables on $(\Omega, \mathcal{F})$. If $\hat \gamma$ and $\hat r$ are well-defined random variables on $(\Omega, \mathcal{F})$, then $\hat h_{\hat \gamma, \hat r}$ is an $(C(S), \bor(C(S)))$-valued measurable function on $(\Omega, \mathcal{F})$ by Lemma 34.

By Lemma \ref{lGaussComp}, the supremum in the definition of $(\hat \gamma, \hat r)$ is at most $40 I_\infty (g, \gamma, r)$ for an appropriate value of $t$. The definition of $(\hat \gamma, \hat r)$ then gives us control over the squared $L^2 (P_n)$ norm of $\hat h_{\hat \gamma, \hat r} - \hat h_{\hat \gamma \wedge \gamma, \hat r \vee r}$. We can control the squared $L^2 (P_n)$ norm of $\hat h_{\hat \gamma \wedge \gamma, \hat r \vee r} - h_{\gamma, r}$ using Lemma \ref{lGaussComp}. In both cases, we use the boundedness of $\Gamma$ when controlling the squared $L^2 (P_n)$ norm before clipping the estimators using $V$. Extra terms appear when moving from bounds on the squared $L^2 (P_n)$ norm to bounds on the squared $L^2 (P)$ norm using Lemma 31. We must then control terms of order $\hat \gamma^{-d/2} \hat r / n^{1/2}$ using \eqref{eGaussCrit} and the definition of $(\hat \gamma, \hat r)$. Combining the bounds gives a bound on the squared $L^2 (P)$ norm of $V h_{\hat \gamma, \hat r} - V h_{\gamma, r}$. Applying
\begin{equation*}
\lVert V \hat h_{\hat r} - g \rVert_{L^2 (P)}^2 \leq 2 \lVert V h_{\hat \gamma, \hat r} - V h_{\gamma, r} \rVert_{L^2 (P)}^2 + 2 \lVert V h_{\gamma, r} - g \rVert_{L^2 (P)}^2
\end{equation*}
leads to Theorem \ref{tGaussLep} on page \pageref{tGaussLep}. Comparisons between $(\hat r, \hat \gamma)$, $(r, \gamma)$ and $(\hat \gamma \wedge \gamma, \hat r \vee r)$ are demonstrated in Figure \ref{fParComp} for two different values of $(r, \gamma)$.

\usetikzlibrary{shapes}
\begin{figure}[h]
\begin{center}
\begin{tikzpicture}
\begin{axis}[axis lines=middle, xmax=20, xmin=0, ymax=10, ymin=0, xlabel={\large $r$}, ylabel={\large $\gamma$}, x=0.5cm, y=0.5cm, xmajorticks=false, ytick={2,8}, yticklabels={\large $u$,\large $v$}, axis line style=thick,every tick/.style={thick}]
\node[label={0:{$r$}}] at (axis cs:0,20) {};
\node[label={0:{$(\hat r, \hat \gamma)$}}, circle, fill, inner sep=1.5pt] at (axis cs:8,5) {};
\node[label={0:{$(r_1, \gamma_1)$}}, rectangle, fill, inner sep=2pt] at (axis cs:8/3,3) {};
\node[label={0:{$(r_2, \gamma_2)$}}, diamond, fill, inner sep=1.5pt] at (axis cs:40/3,7) {};
\node[label={0:{$(\hat r \vee r_1, \hat \gamma \wedge \gamma_1)$}}, rectangle, fill, inner sep=2pt] at (axis cs:8,3) {};
\node[label={0:{$(\hat r \vee r_2, \hat \gamma \wedge \gamma_2)$}}, diamond, fill, inner sep=1.5pt] at (axis cs:40/3,5) {};
\end{axis}
\end{tikzpicture}
\end{center}
\caption{A demonstration of the parameter comparisons made in the proof of Theorem \ref{tGaussLep}}
\label{fParComp}
\end{figure}
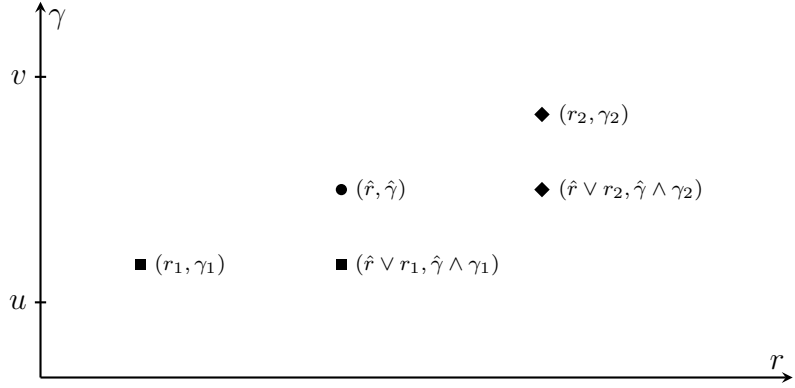

We now combine Theorems \ref{tVaryBound} and \ref{tGaussLep}. In the following,
let $A_{3, t} \in \mathcal{F}$ be a event such that $\prob(A_{3, t}) \geq 1 - e^{-t}$ and on which
\begin{equation*}
\lVert \hat h_{k, r} - h_{k, r} \rVert_{L^2 (P_n)}^2 \leq \frac{21 J \lVert k \rVert_{\diag}^{1/2} \sigma r t^{1/2}}{n^{1/2}} + 4 \lVert h_{k, r} - g \rVert_\infty^2
\end{equation*}
simultaneously for all $k \in \mathcal{K}$, all $r \geq 0$ and all $h_{k, r} \in r B_k$. By Lemma 29 in \cite{page2018supplement} such an $A_{3,t}$ exists. Furthermore, let
\begin{align*}
A_{4,t} = \Bigl\{\sup_{f_1, f_2 \in r B_k} &\left \lvert \lVert V f_1 - V f_2 \rVert_{L^2 (P_n)}^2 - \lVert V f_1 - V f_2 \rVert_{L^2 (P)}^2 \right \rvert \\ 
&\leq \frac{151 J \lVert k \rVert_{\diag}^{1/2} C r t^{1/2}}{n^{1/2}} + \frac{8 \lVert k \rVert_{\diag}^{1/2} C r t}{3 n}, \text{  for all } k \in \mathcal{K} \text{ and all } r \geq 0
\Bigr\}.
\end{align*}
With these definition in place the following holds.

\begin{theorem} \label{tGaussLepBound}
Assume \ref{g1}, \ref{Y} and \ref{K2}. Let $\tau \geq 84 J \sigma$, $\nu > 0$ and
\begin{equation*}
t = \left ( \frac{\tau}{84 J \sigma} \right )^2 \geq 1.
\end{equation*}
On the set $A_{3, t} \cap A_{4, t} \in \mathcal{F}$, for which $\prob(A_{3, t} \cap A_{4, t}) \geq 1 - 2 e^{-t}$, we have
\begin{equation*}
\lVert V \hat h_{\gamma, \hat r} - g \rVert_{L^2 (P)}^2 \leq \inf_{\gamma \in \Gamma} \inf_{r \in R} \left ( (1 + D_1 \tau n^{-1/2}) (D_2 \tau \gamma^{-d/2} r n^{-1/2} + D_3 I_\infty (g, \gamma, r)) \right )
\end{equation*}
for constants $D_1, D_2, D_3 > 0$ not depending on $\tau$, $\gamma$, $r$ or $n$.
\end{theorem}

We can obtain rates of convergence for our estimator $V \hat h_{\hat \gamma, \hat r}$ if we make an assumption about how well $g$ can be approximated by elements of $H_\alpha$ for $\alpha \in [u, v]$. Let us assume
\begin{enumerate}[label={($g$3)}]
\item \hfil $g \in [C(S), H_\alpha]_{\beta, \infty}$ with norm at most $B$ for $\alpha \in [u, v]$, $\beta \in (0, 1)$ and $B > 0$. \label{g3}
\end{enumerate}
The assumption \ref{g3}, together with Lemma \ref{lApproxInter}, give
\begin{equation}
I_\infty (g, \alpha, r) \leq \frac{B^{2/(1-\beta)}}{r^{ 2 \beta/(1-\beta)}} \label{eGaussApproxInter}
\end{equation}
for $r > 0$. In order for us to apply Theorem \ref{tGaussLepBound} to this setting, we need to make assumptions on $\Gamma$ and $R$. We assume either \ref{R1} and
\begin{enumerate}[label={($\Gamma$1)}]
\item \hfil $\Gamma = [u, v]$, \label{Gamma1}
\end{enumerate}
or \ref{R2} and
\begin{enumerate}[label={($\Gamma$2)}]
\item \hfil $\Gamma = \{ u c^i : 0 \leq i \leq L - 1 \} \cup \{v\}$ for $c > 1$ and $L = \lceil \log(v/u)/\log(c) \rceil$. \label{Gamma2}
\end{enumerate}
The assumptions \ref{R1} and \ref{Gamma1} are mainly of theoretical interest and would make it difficult to calculate $(\hat \gamma, \hat r)$ in practice. The estimator $(\hat \gamma, \hat r)$ can be computed under the assumptions \ref{R2} and \ref{Gamma2}, since in this case $R$ and $\Gamma$ are finite. We obtain a high-probability bound on a fixed quantile of the squared $L^2 (P)$ error of $V \hat h_{\hat r, \hat \gamma}$ of order $t^{1/2} n^{-\beta/(1+\beta)}$ with probability at least $1 - e^{- t}$ when $\tau$ is an appropriate multiple of $t^{1/2}$.

\begin{corollary} \label{tInterGaussLepBound}
Assume \ref{g1}, \ref{g3}, \ref{Y} and \ref{K2}. Let $\tau \geq 84 J \sigma$, $\nu > 0$ and
\begin{equation*}
t = \left ( \frac{\tau}{84 J \sigma} \right )^2 \geq 1.
\end{equation*}
Also assume \ref{R1}, \ref{Gamma1}, \ref{hr} and \ref{hgamma}, or \ref{R2} and \ref{Gamma2}.  On the set $A_{3, t} \cap A_{4, t} \in \mathcal{F}$, for which $\prob(A_{3, t} \cap A_{4, t}) \geq 1 - 2 e^{-t}$, we have
\begin{equation*}
\lVert V \hat h_{\hat \gamma, \hat r} -  g \rVert_{L^2 (P)}^2 \leq D_1 \tau n^{-\beta/(1+\beta)} + D_2 \tau^2 n^{-(1+3\beta)/(2(1+\beta))}
\end{equation*}
for constants $D_1, D_2 > 0$ not depending on $n$ or $\tau$.
\end{corollary}
From the high probability bound we obtain a bound on the expected squared $L^2(P)$ error similarly to 
Corollary \ref{InterLepExpBound}.
\begin{corollary} \label{InterGaussLepExpBound}
Assume \ref{g1}, \ref{g3}, \ref{Y}, \ref{K2} and either \ref{R1}, \ref{Gamma1}, \ref{hr} and \ref{hgamma}, or \ref{R2} and \ref{Gamma2}. We have 
\[
\expect(\lVert V \hat h_{\hat \gamma, \hat r} -  g \rVert_{L^2 (P)}^2) \leq D n^{-\beta/(1+\beta)}
\]
for constant $D$ not depending on $n$.
\end{corollary}

\section{Discussion}

In this paper, we show how the Goldenshluger--Lepski method can be applied when performing regression over an RKHS $H$, which is separable with a bounded and measurable kernel $k$, or a collection of such RKHSs. We produce an adaptive estimator from a collection of clipped versions of least-squares estimators which are constrained to lie in a ball of predefined radius in $H$. Since the $L^2 (P)$ norm is unknown, we use the $L^2 (P_n)$ norm when calculating the pairwise comparisons for the proxy for the unknown bias of this collection of non-adaptive estimators. When $H$ is fixed, our estimator need only adapt to the radius of the ball in $H$. However, when $H$ comes from a collection of RKHSs with Gaussian kernels, the estimator must also adapt to the width parameter of the kernel. As far as we are aware, this is the first time that the Goldenshluger--Lepski method has been applied in the context of RKHS regression. In order to apply the Goldenshluger--Lepski method in this context, we must provide a majorant by controlling all of the non-adaptive estimators simultaneously, extending the results of \cite{page2017ivanov}.

By assuming that the regression function lies in an interpolation space between $C(S)$ and $H$ parametrised by $\beta$, we obtain a bound on a fixed quantile of the squared $L^2 (P)$ error of our adaptive estimator of order $n^{- \beta/(1+\beta)}$. This is true for both the case in which $H$ is fixed and the case in which $H$ comes from a collection of RKHSs with Gaussian kernels. The order $n^{- \beta/(1+\beta)}$ for the squared $L^2 (P)$ error of the adaptive estimators matches the order of the smallest bounds obtained in \cite{page2017ivanov} for the squared $L^2 (P)$ error of the non-adaptive estimators. In the sense discussed in \cite{page2017ivanov}, this order is the optimal power of $n$ if we make the slightly weaker assumption that the regression function is an element of the interpolation space between $L^2 (P)$ and $H$ parametrised by $\beta$.

For the case in which $H$ comes from a collection of RKHSs with Gaussian kernels, our current results rely on the boundedness of the set $\Gamma$ of width parameters of the kernels. This is somewhat limiting as allowing the width parameter to tend to 0 as $n$ tends to infinity would allow us to estimate a greater collection of functions. We hope that in the future the analysis in the proof of Theorem \ref{tGaussLep} can be extended to allow for such flexibility.

The results in this paper warrant the investigation of whether it is possible to extend the use of the Goldenshluger--Lepski method from the case in which $H$ comes from a collection of RKHSs with Gaussian kernels to other cases. The analysis in this paper relies on the fact that the closed unit ball of the RKHS generated by a Gaussian kernel increases as the width of the kernel decreases. It may be possible to apply a similar analysis to other situations in which $H$ belongs to a collection of RKHSs which also exhibit this nestedness property. If the RKHSs did not exhibit this property, then a new form of analysis would be necessary to apply the Goldenshluger--Lepski method. In particular, we would need a new criterion for deciding on the smoothness of the non-adaptive estimators when performing the pairwise comparisons.

\appendix

\section{Proof of the Goldenshluger--Lepski Method for a Fixed RKHS} \label{sec:app_fixed}
This section is composed of two subsections: in Subsection \ref{sec:single_app_aux} we derive a number of auxiliary results and in Subsection \ref{sec:single_app_main} we provide proofs for our main results.

\subsection{Auxiliary Results} \label{sec:single_app_aux}
The following combines parts of Lemmas 3 and 32 of \cite{page2017ivanov}.

\begin{lemma} \label{lEst}
Assume \ref{H}. We have that $\hat h_r$ is a $(H, \bor(H))$-valued measurable function on $(\Omega \times [0, \infty), \mathcal{F} \otimes \bor([0, \infty)))$, where $r$ varies in $[0, \infty)$. Furthermore, $\lVert \hat h_r - \hat h_s \rVert_H^2 \leq \lvert r^2 - s^2 \rvert$ for $r, s \in [0, \infty)$.
\end{lemma}

We bound the distance between $\hat h_r$ and $\hat h_s$ in the $L^2(P_n)$ norm for $s \geq r \geq 0$ to prove the following Lemma.

\begin{lemma} \label{lComp}
Assume \ref{Y} and \ref{H}. Let $t \geq 1$ and recall the definition of $A_{1, t}$ from Lemma 22. On the set $A_{1, t} \in \mathcal{F}$, for which $\prob(A_{1, t}) \geq 1 - e^{-t}$, we have
\begin{equation*}
\lVert \hat h_r - \hat h_s \rVert_{L^2 (P_n)}^2 \leq \frac{80 \lVert k \rVert_{\diag}^{1/2} \sigma (r + s) t^{1/2}}{n^{1/2}} + 40 I_\infty (g, r)
\end{equation*}
simultaneously for all $s \geq r \geq 0$.
\end{lemma}
\begin{proof}
By Lemma 22, we have
\begin{align*}
\lVert \hat h_r - \hat h_s \rVert_{L^2 (P_n)}^2 &\leq 4 \lVert \hat h_r - h_r \rVert_{L^2 (P_n)}^2 + 4 \lVert h_r - g \rVert_{L^2 (P_n)}^2 \\
&+ 4 \lVert g - h_s \rVert_{L^2 (P_n)}^2 + 4 \lVert h_s - \hat h_s \rVert_{L^2 (P_n)}^2 \\
&\leq \frac{80 \lVert k \rVert_{\diag}^{1/2} \sigma (r + s) t^{1/2}}{n^{1/2}} + 20 \lVert h_r - g \rVert_\infty^2 + 20 \lVert h_s - g \rVert_\infty^2
\end{align*}
for all $r, s \geq 0$ and all $h_r \in r B_H, h_s \in s B_H$. Taking an infimum over $h_r \in r B_H$ and $h_s \in s B_H$ gives
\begin{equation*}
\lVert \hat h_r - \hat h_s \rVert_{L^2 (P_n)}^2 \leq \frac{80 \lVert k \rVert_{\diag}^{1/2} \sigma (r + s) t^{1/2}}{n^{1/2}} + 20 I_\infty (g, r) + 20 I_\infty (g, s).
\end{equation*}
The result follows. 
\end{proof}

Next, we show that the estimator $\hat r$ is well-defined. 

\begin{lemma} \label{lEstWD}
Let $\hat r$ be the infimum of all points attaining the minimum in \eqref{eEst}. Then $\hat r$ is well-defined.
\end{lemma}
\begin{proof}
Let $K$ be the $n \times n$ symmetric matrix with $K_{i, j} = k(X_i, X_j)$. By Lemma 3 of \cite{page2017ivanov}, we have that $K$ is an $(\real^{n \times n}, \bor(\real^{n \times n}))$-valued measurable matrix on $(\Omega, \mathcal{F})$ and that there exist an orthogonal matrix $A$ and a diagonal matrix $D$ which are both $(\real^{n \times n}, \bor(\real^{n \times n}))$-valued measurable matrices on $(\Omega, \mathcal{F})$ such that $K = A D A^{\trans}$. Furthermore, we can demand that the diagonal entries of $D$ are non-negative and non-increasing. Let $m = \rk K$ and
\begin{equation*}
\rho = \left ( \sum_{i = 1}^m D_{i, i}^{-1} (A^{\trans} Y)_i^2 \right )^{1/2},
\end{equation*}
which are random variables on $(\Omega, \mathcal{F})$. By Lemma 3 of \cite{page2017ivanov}, we have that $\hat h_r$ is constant in $r$ for $r \geq \rho$. Hence,
\begin{align}
&\inf_{r \in R} \left ( \sup_{s \in R, s \geq r} \left ( \lVert \hat h_r - \hat h_s \rVert_{L^2 (P_n)}^2 - \frac{\tau (r + s)}{n^{1/2}} \right ) + \frac{2 (1 + \nu) \tau r}{n^{1/2}} \right ) \nonumber \\
= &\inf_{r \in R \cap [0, \rho]} \left ( \sup_{s \in R, s \geq r} \left ( \lVert \hat h_r - \hat h_s \rVert_{L^2 (P_n)}^2 - \frac{\tau (r + s)}{n^{1/2}} \right ) + \frac{2 (1 + \nu) \tau r}{n^{1/2}} \right ). \label{eEstInf}
\end{align}
 By Lemma \ref{lEst}, we have
\begin{equation*}
\lVert \hat h_r - \hat h_s \rVert_{L^2 (P_n)}^2 - \frac{\tau (r + s)}{n^{1/2}}
\end{equation*}
is continuous in $r$ for all $s \in R$ such that $s \geq r$. The supremum of a collection of lower semi-continuous functions is lower semi-continuous. Therefore,
\begin{equation*}
\sup_{s \in R, s \geq r} \left ( \lVert \hat h_r - \hat h_s \rVert_{L^2 (P_n)}^2 - \frac{\tau (r + s)}{n^{1/2}} \right ) + \frac{2 (1 + \nu) \tau r}{n^{1/2}}
\end{equation*}
is lower semi-continuous in $r$. Hence, the infimum \eqref{eEstInf} is attained as it is the infimum of a lower semi-continuous function on a compact set. By lower semi-continuity, $\hat r$ also attains the infimum and is well-defined. 
\end{proof}

\subsection{Main Results} \label{sec:single_app_main}
We use the Goldenshluger--Lepski method to prove the following Theorem.

%{\bf Proof of Theorem \ref{tLep}}
\begin{theorem} \label{tLep}
Assume \ref{Y}, \ref{H} and \ref{hr}. Let $\tau \geq 80 \lVert k \rVert_{\diag}^{1/2} \sigma$, $\nu > 0$ and
\begin{equation*}
t = \left ( \frac{\tau}{80 \lVert k \rVert_{\diag}^{1/2} \sigma} \right )^2 \geq 1.
\end{equation*}
Recall the definitions of $A_{1, t}$ and $A_{2, t}$ from Lemmas 22 and 24. On the set $A_{1, t} \cap A_{2, t} \in \mathcal{F}$, for which $\prob(A_{1, t} \cap A_{2, t}) \geq 1 - 2 e^{-t}$, we have
\begin{equation*}
\lVert V \hat h_{\hat r} - g \rVert_{L^2 (P)}^2
\end{equation*}
is at most
\begin{align*}
&\inf_{r \in R}\left ( \max \left \{ \frac{2 \tau r}{n^{1/2}} + \left ( \frac{1}{\nu} + \frac{97 C}{80 \sigma \nu} + \frac{C \tau}{2400 \lVert k \rVert_{\diag}^{1/2} \sigma^2 \nu n^{1/2}} \right ) \left ( 40 I_\infty (g, r) + \frac{2 (1 + \nu) \tau r}{n^{1/2}} \right ), \right. \right. \\
& \left. \left. \frac{4 (2 + \nu) \tau r}{n^{1/2}} + \frac{97 C \tau r}{40 \sigma n^{1/2}} + \frac{C \tau^2 r}{1200 \lVert k \rVert_{\diag}^{1/2} \sigma^2 n} \right \} + 80 I_\infty (g, r) + 2 \lVert V \hat h_r - g \rVert_{L^2 (P)}^2 \right ).
\end{align*}
\end{theorem}
\begin{proof}
Since we assume \ref{Y} and \ref{H}, we find that Lemma 22 holds, which implies that Lemma \ref{lComp} holds. By our choice of $t$, we have
\begin{equation}
\lVert \hat h_r - \hat h_s \rVert_{L^2 (P_n)}^2 \leq \frac{\tau (r + s)}{n^{1/2}} + 40 I_\infty (g, r) \label{eComp}
\end{equation}
simultaneously for all $s, r \in R$ such that $s \geq r \geq 0$. Fix $r \in R$ and suppose that $\hat r \leq r$. By the definition of $\hat r$ in \eqref{eEst} and \eqref{eComp}, we have
\begin{align*}
\lVert \hat h_{\hat r} - \hat h_r \rVert_{L^2 (P_n)}^2 &= \lVert \hat h_{\hat r} - \hat h_r \rVert_{L^2 (P_n)}^2 - \frac{\tau (\hat r + r)}{n^{1/2}} + \frac{\tau (\hat r + r)}{n^{1/2}} \\
&\leq \sup_{s \in R, s \geq \hat r} \left ( \lVert \hat h_{\hat r} - \hat h_s \rVert_{L^2 (P_n)}^2 - \frac{\tau (\hat r + s)}{n^{1/2}} \right ) + \frac{2 \tau r}{n^{1/2}} \\
&\leq \sup_{s \in R, s \geq r} \left ( \lVert \hat h_r - \hat h_s \rVert_{L^2 (P_n)}^2 - \frac{\tau (r + s)}{n^{1/2}} \right ) + \frac{2 (2 + \nu) \tau r}{n^{1/2}} - \frac{2 (1 + \nu) \tau \hat r}{n^{1/2}} \\
&\leq 40 I_\infty (g, r) + \frac{2 (2 + \nu) \tau r}{n^{1/2}}.
\end{align*}
This shows
\begin{equation*}
\lVert V \hat h_{\hat r} - V \hat h_r \rVert_{L^2 (P_n)}^2 \leq 40 I_\infty (g, r) + \frac{2 (2 + \nu) \tau r}{n^{1/2}},
\end{equation*}
and it follows from Lemma 24 and our choice of $t$ that
\begin{equation*}
\lVert V \hat h_{\hat r} - V \hat h_r \rVert_{L^2 (P)}^2 \leq 40 I_\infty (g, r) + \frac{2 (2 + \nu) \tau r}{n^{1/2}} + \frac{97 C \tau r}{80 \sigma n^{1/2}} + \frac{C \tau^2 r}{2400 \lVert k \rVert_{\diag}^{1/2} \sigma^2 n}.
\end{equation*}
Hence,
\begin{align*}
&\lVert V \hat h_{\hat r} - g \rVert_{L^2 (P)}^2 \\
\leq \; &2 \lVert V \hat h_{\hat r} - V \hat h_r \rVert_{L^2 (P)}^2 + 2 \lVert V \hat h_r - g \rVert_{L^2 (P)}^2 \\
\leq \; &80 I_\infty (g, r) + \frac{4 (2 + \nu) \tau r}{n^{1/2}} + \frac{97 C \tau r}{40 \sigma n^{1/2}} + \frac{C \tau^2 r}{1200 \lVert k \rVert_{\diag}^{1/2} \sigma^2 n} + 2 \lVert V \hat h_r - g \rVert_{L^2 (P)}^2.
\end{align*}

Now suppose instead that $\hat r \geq r$. Since \eqref{eComp} holds simultaneously for all $s, r \in R$ such that $s \geq r \geq 0$, we have
\begin{equation*}
\lVert \hat h_{\hat r} - \hat h_r \rVert_{L^2 (P_n)}^2 \leq \frac{\tau (r + \hat r)}{n^{1/2}} + 40 I_\infty (g, r).
\end{equation*}
This shows
\begin{equation*}
\lVert V \hat h_{\hat r} - V \hat h_r \rVert_{L^2 (P_n)}^2 \leq \frac{\tau r}{n^{1/2}} + 40 I_\infty (g, r) + \frac{\tau \hat r}{n^{1/2}},
\end{equation*}
and it follows from Lemma 24 that
\begin{align*}
&\lVert V \hat h_{\hat r} - V \hat h_r \rVert_{L^2 (P)}^2 \\
\leq \; &\frac{\tau r}{n^{1/2}} + 40 I_\infty (g, r) + \frac{\tau \hat r}{n^{1/2}} + \frac{97 C \tau \hat r}{80 \sigma n^{1/2}} + \frac{C \tau^2 \hat r}{2400 \lVert k \rVert_{\diag}^{1/2} \sigma^2 n} \\
= \; &\frac{\tau r}{n^{1/2}} + 40 I_\infty (g, r) + \left ( \frac{1}{2 \nu} + \frac{97 C}{160 \sigma \nu} + \frac{C \tau}{4800 \lVert k \rVert_{\diag}^{1/2} \sigma^2 \nu n^{1/2}} \right ) \frac{2 \nu \tau \hat r}{n^{1/2}}.
\end{align*}
By \eqref{eCrit}, the definition of $\hat r$ in \eqref{eEst} and \eqref{eComp}, we have
\begin{align*}
\frac{2 \nu \tau \hat r}{n^{1/2}} &\leq \sup_{s \in R, s \geq \hat r} \left ( \lVert \hat h_{\hat r} - \hat h_s \rVert_{L^2 (P_n)}^2 - \frac{\tau (\hat r + s)}{n^{1/2}} \right ) + \frac{2 (1 + \nu) \tau \hat r}{n^{1/2}} \\
&\leq \sup_{s \in R, s \geq r} \left ( \lVert \hat h_r - \hat h_s \rVert_{L^2 (P_n)}^2 - \frac{\tau (r + s)}{n^{1/2}} \right ) + \frac{2 (1 + \nu) \tau r}{n^{1/2}} \\
&\leq 40 I_\infty (g, r) + \frac{2 (1 + \nu) \tau r}{n^{1/2}}.
\end{align*}
Hence,
\begin{align*}
&\lVert V \hat h_{\hat r} - g \rVert_{L^2 (P)}^2 \\
\leq \; &2 \lVert V \hat h_{\hat r} - V \hat h_r \rVert_{L^2 (P)}^2 + 2 \lVert V \hat h_r - g \rVert_{L^2 (P)}^2 \\
\leq \; &\frac{2 \tau r}{n^{1/2}} + 80 I_\infty (g, r) + \left ( \frac{1}{\nu} + \frac{97 C}{80 \sigma \nu} + \frac{C \tau}{2400 \lVert k \rVert_{\diag}^{1/2} \sigma^2 \nu n^{1/2}} \right ) \left ( 40 I_\infty (g, r) + \frac{2 (1 + \nu) \tau r}{n^{1/2}} \right ) \\
+ \; &2 \lVert V \hat h_r - g \rVert_{L^2 (P)}^2.
\end{align*}
The result follows. 
\end{proof}

We assume \ref{g1} to bound the distance between $V \hat h_{\hat r}$ and $g$ in the $L^2(P)$ norm and prove Theorem \ref{tLepBound}.

{\bf Proof of Theorem \ref{tLepBound}}
By Theorem \ref{tLep}, we have
\begin{equation*}
\lVert V \hat h_{\hat r} - g \rVert_{L^2 (P)}^2 \leq \inf_{r \in R} \left ( (1 + D_4 \tau n^{-1/2}) (D_5 \tau r n^{-1/2} + D_6 I_\infty (g, r)) + 2 \lVert V \hat h_r - g \rVert_{L^2 (P)}^2 \right )
\end{equation*}
for some constants $D_4, D_5, D_6 > 0$ not depending on $\tau$, $r$ or $n$. By Theorem \ref{tBound}, we have
\begin{align*}
\lVert V \hat h_r - g \rVert_{L^2 (P)}^2 &\leq \frac{(97 C + 20 \sigma) \tau r}{40 \sigma n^{1/2}} + \frac{C \tau^2 r}{1200 \lVert k \rVert_{\diag}^{1/2} \sigma^2 n} + 10 I_\infty (g, r) \\
&\leq D_7 \tau r n^{-1/2} + D_8 \tau^2 r n^{-1} + 10 I_\infty (g, r).
\end{align*}
for all $r \in R$, for some constants $D_7, D_8 > 0$ not depending on $\tau$, $r$ or $n$. This gives
\begin{align*}
\lVert V \hat h_{\hat r} - g \rVert_{L^2 (P)}^2 \leq \inf_{r \in R} &\left ( (1 + D_4 \tau n^{-1/2}) (D_5 \tau r n^{-1/2} + D_6 I_\infty (g, r)) \right . \\
&+ \left . 2 D_7 \tau r n^{-1/2} + 2 D_8 \tau^2 r n^{-1} + 20 I_\infty (g, r) \right ).
\end{align*}
Hence, the result follows with
\begin{equation*}
D_1 = \frac{D_4 D_5 + 2 D_8}{D_5 + 2 D_7}, \; D_2 = D_5 + 2 D_7, \; D_3 = D_6 + 20.
\end{equation*} \hfill \BlackBox

We assume \ref{g2} to prove Theorem \ref{tInterLepBound}.

{\bf Proof of Corollary \ref{tInterLepBound}}
If we assume \ref{R1}, then $r = a n^{(1-\beta)/(2 (1+\beta))} \in R$ and
\begin{align*}
\lVert V \hat h_{\hat r} - g \rVert_{L^2 (P)}^2 &\leq (1 + D_3 \tau n^{-1/2}) (D_4 \tau r n^{-1/2} + D_5 I_\infty (g, r)) \\
&\leq (1 + D_3 \tau n^{-1/2}) \left ( D_4 \tau a n^{- \beta/(1+\beta)} + \frac{D_5 B^{2/(1-\beta)}}{a^{2 \beta/(1-\beta)} n^{\beta/(1+\beta)}} \right )
\end{align*}
for some constants $D_3, D_4, D_5 > 0$ not depending on $n$ or $\tau$ by Theorem \ref{tLepBound} and \eqref{eApproxInter}. If we assume \ref{R2}, then there is a unique $r \in R$ such that
\begin{equation*}
a n^{(1-\beta)/(2 (1+\beta))} \leq r < a n^{(1-\beta)/(2 (1+\beta))} + b
\end{equation*}
and
\begin{align*}
&\lVert V \hat h_{\hat r} - g \rVert_{L^2 (P)}^2 \\
\leq \; &(1 + D_3 \tau n^{-1/2}) (D_4 \tau r n^{-1/2} + D_5 I_\infty (g, r)) \\
\leq \; &(1 + D_3 \tau n^{-1/2}) \left ( D_4 \tau (a n^{(1-\beta)/(2 (1+\beta))} + b) n^{- 1/2} + \frac{D_5 B^{2/(1-\beta)}}{a^{2 \beta/(1-\beta)} n^{\beta/(1+\beta)}} \right )
\end{align*}
by Theorem \ref{tLepBound} and \eqref{eApproxInter}. In either case,
\begin{equation*}
\lVert V \hat h_{\hat r} -  g \rVert_{L^2 (P)}^2 \leq D_1 \tau n^{-\beta/(1+\beta)} + D_2 \tau^2 n^{-(1+3\beta)/(2(1+\beta))}
\end{equation*}
for some constants $D_1, D_2 > 0$ not depending on $n$ or $\tau$. \hfill \BlackBox

{\bf Proof of Corollary \ref{InterLepExpBound}} Under the stated assumptions we can apply Theorem \ref{tLepBound}. Let $D_1' = 80 \lVert k \rVert_{\diag}^{1/2} \sigma D_1 n^{-\beta/(1+\beta)}$ and
$D_2' =  6400 \lVert k \rVert_{\diag} D_2 \sigma^2  n^{-(1+3\beta)/(2(1+\beta))}$ with constants $D_1,D_2$ from Theorem \ref{tLepBound}. Furthermore, let 
$\phi(t) = D_1' t^{1/2}   + D_2' t$ then   
\begin{align*}
&\expect(\lVert V \hat h_{\hat r} -  g \rVert_{L^2 (P)}^2) \\
=& \int_{0}^\infty  \prob( \lVert V \hat h_{\hat r} -  g \rVert_{L^2 (P)}^2 \geq a) \, da \\
=& \int_{0}^\infty  \prob( \lVert V \hat h_{\hat r} -  g \rVert_{L^2 (P)}^2 \geq \phi(t)) (D_1' t^{1/2} + D_2') \, dt \\
\leq&  (\pi^{1/2} +1) D_1' + 3D_2',   
\end{align*}
using the bound on $\prob( \lVert V \hat h_{\hat r} -  g \rVert_{L^2 (P)}^2 \geq \phi(t))$ for $t\geq 1$. 
The result follows by letting $D = 80 (\pi^{1/2} +1)  \lVert k \rVert_{\diag}^{1/2} \sigma D_1 +  19200 \lVert k \rVert_{\diag} D_2 \sigma^2$.
\hfill \BlackBox

\section{Proof of the Goldenshluger--Lepski Method for a Collection of RKHSs with Gaussian Kernels}
This section is composed of two subsections, one containing auxiliary results (Subsection \ref{sec:multiple_app_aux}) and one containing the proofs of our main results  (Subsection \ref{sec:multiple_app_main}).

\subsection{Auxiliary Results} \label{sec:multiple_app_aux}
We bound the distance between $\hat h_{\gamma, r}$ and $\hat h_{\eta, s}$ in the $L^2(P_n)$ norm for $\gamma, \eta \in \Gamma$ with $\eta \leq \gamma$ and $s \geq r \geq 0$ to prove the following  Lemma. 

%{\bf Proof of Lemma \ref{lGaussComp}}
\begin{lemma} \label{lGaussComp}
Assume \ref{Y} and \ref{K2}. Let $t \geq 1$ and recall the definition of $A_{3, t}$ from Lemma 29. On the set $A_{3, t} \in \mathcal{F}$, for which $\prob(A_{3, t}) \geq 1 - e^{-t}$, we have
\begin{equation*}
\lVert \hat h_{\gamma, r} - \hat h_{\eta, s} \rVert_{L^2 (P_n)}^2 \leq \frac{84 J \sigma (\gamma^{-d/2} r + \eta^{-d/2} s) t^{1/2}}{n^{1/2}} + 40 I_\infty (g, \gamma, r)
\end{equation*}
simultaneously for all $\gamma, \eta \in \Gamma$ such that $\eta \leq \gamma$ and all $s \geq r \geq 0$.
\end{lemma}
\begin{proof}
By Lemma 29, we have
\begin{align*}
\lVert \hat h_{\gamma, r} - \hat h_{\eta, s} \rVert_{L^2 (P_n)}^2 &\leq 4 \lVert \hat h_{\gamma, r} - h_{\gamma, r} \rVert_{L^2 (P_n)}^2 + 4 \lVert h_{\gamma, r} - g \rVert_{L^2 (P_n)}^2 \\
&+ 4 \lVert g - h_{\eta, s} \rVert_{L^2 (P_n)}^2 + 4 \lVert h_{\eta, s} - \hat h_{\eta, s} \rVert_{L^2 (P_n)}^2 \\
&\leq \frac{84 J \sigma (\gamma^{-d/2} r + \eta^{-d/2} s) t^{1/2}}{n^{1/2}} + 20 \lVert h_{\gamma, r} - g \rVert_\infty^2 + 20 \lVert h_{\eta, s} - g \rVert_\infty^2
\end{align*}
for all $\gamma, \eta \in \Gamma$, all $r, s \geq 0$ and all $h_{\gamma, r} \in r B_\gamma, h_{\eta, s} \in s B_\eta$. Taking an infimum over $h_{\gamma, r} \in r B_\gamma$ and $h_{\eta, s} \in s B_\eta$ gives
\begin{equation*}
\lVert \hat h_{\gamma, r} - \hat h_{\eta, s} \rVert_{L^2 (P_n)}^2 \leq \frac{84 J \sigma (\gamma^{-d/2} r + \eta^{-d/2} s) t^{1/2}}{n^{1/2}} + 20 I_\infty (g, \gamma, r) + 20 I_\infty (g, \eta, s).
\end{equation*}
The result follows from Lemma \ref{lNest}. %\hfill \BlackBox
\end{proof}

\subsection{Main Results} \label{sec:multiple_app_main}
We use the Goldenshluger--Lepski method to prove the following Theorem.

\begin{theorem} \label{tGaussLep}
Assume \ref{Y} and \ref{K2}. Let $\tau \geq 84 J \sigma$, $\nu > 0$ and
\begin{equation*}
t = \left ( \frac{\tau}{84 J \sigma} \right )^2 \geq 1.
\end{equation*}
Recall the definitions of $A_{3, t}$ and $A_{4, t}$ from Lemmas 29 and 31. On the set $A_{3, t} \cap A_{4, t} \in \mathcal{F}$, for which $\prob(A_{3, t} \cap A_{4, t}) \geq 1 - 2 e^{-t}$, we have
\begin{equation*}
\lVert V \hat h_{\hat \gamma, \hat r} - g \rVert_{L^2 (P)}^2
\end{equation*}
is at most
\begin{align*}
&\inf_{\gamma \in \Gamma} \inf_{r \in R} \left ( 320 I_\infty (g, \gamma, r) + \frac{4 v^{d/2} (5 + 2 \nu) \tau \gamma^{-d/2} r}{u^{d/2} n^{1/2}} + \frac{302 C v^{d/2} \tau \gamma^{-d/2} r}{21 u^{d/2} \sigma n^{1/2}} + \frac{4 C v^{d/2} \tau^2 \gamma^{-d/2} r}{1323 J^2 u^{d/2} \sigma^2 n} \right. \\
&+ \left ( \frac{12 v^{d/2}}{u^{d/2} \nu} + \frac{302 C v^{d/2}}{21 u^{d/2} \sigma \nu} + \frac{4 C v^{d/2} \tau}{1323 J^2 u^{d/2} \sigma^2 \nu n^{1/2}} \right ) \left ( 20 I_\infty (g, \gamma, r) + \frac{(1 + \nu) \tau \gamma^{-d/2} r}{n^{1/2}} \right ) \\
&+ \left. 2 \lVert V \hat h_{\gamma, r} - g \rVert_{L_2 (P)}^2 \right ).
\end{align*}
\end{theorem}
\begin{proof}
Since we assume \ref{Y} and \ref{K2}, which implies \ref{K1}, we find that Lemma 29 holds, which implies that Lemma \ref{lGaussComp} holds. By our choice of $t$, we have
\begin{equation}
\lVert \hat h_{\gamma, r} - \hat h_{\eta, s} \rVert_{L^2 (P_n)}^2 \leq \frac{\tau (\gamma^{-d/2} r + \eta^{-d/2} s)}{n^{1/2}} + 40 I_\infty (g, \gamma, r) \label{eGaussComp}
\end{equation}
simultaneously for all $\gamma, \eta \in \Gamma$ and all $r, s \in R$ such that $\eta \leq \gamma$ and $s \geq r$. Fix $\gamma \in \Gamma$ and $r \in R$. Then
\begin{equation*}
\lVert V \hat h_{\hat \gamma, \hat r} - V \hat h_{\gamma, r} \rVert_{L_2 (P)}^2 \leq 2 \lVert V \hat h_{\hat \gamma, \hat r} - V \hat h_{\hat \gamma \wedge \gamma, \hat r \vee r} \rVert_{L_2 (P)}^2 + 2 \lVert V h_{\hat \gamma \wedge \gamma, \hat r \vee r} - V \hat h_{\gamma, r} \rVert_{L_2 (P)}^2.
\end{equation*}
We now bound the right-hand side. By $\Gamma \subs [u, v]$, the definition of $(\hat \gamma, \hat r)$ in \eqref{eGaussEst} and \eqref{eGaussComp}, we have
\begin{align*}
&\lVert \hat h_{\hat \gamma, \hat r} - \hat h_{\hat \gamma \wedge \gamma, \hat r \vee r} \rVert_{L_2 (P_n)}^2 \\
= \; &\lVert \hat h_{\hat \gamma, \hat r} - \hat h_{\hat \gamma \wedge \gamma, \hat r \vee r} \rVert_{L_2 (P_n)}^2 - \frac{\tau (\hat \gamma^{-d/2} \hat r + (\hat \gamma \wedge \gamma)^{-d/2} (\hat r \vee r))}{n^{1/2}} \\
+ \; &\frac{\tau (\hat \gamma^{-d/2} \hat r + (\hat \gamma \wedge \gamma)^{-d/2} (\hat r + r))}{n^{1/2}} \\
\leq \; &\sup_{\eta \in \Gamma, \eta \leq \hat \gamma} \sup_{s \in R, s \geq \hat r} \left ( \lVert \hat h_{\hat \gamma, \hat r} - \hat h_{\eta, s} \rVert_{L_2 (P_n)}^2 - \frac{\tau \left ( \hat \gamma^{-d/2} \hat r + \eta^{-d/2} s \right )}{n^{1/2}} \right ) \\
+ \; &\frac{\tau (\hat \gamma^{-d/2} \hat r + (v/u)^{d/2} (\hat \gamma^{-d/2} \hat r + \gamma^{-d/2} r))}{n^{1/2}} \\
\leq \; &\sup_{\eta \in \Gamma, \eta \leq \gamma} \sup_{s \in R, s \geq r} \left ( \lVert \hat h_{\gamma, r} - \hat h_{\eta, s} \rVert_{L_2 (P_n)}^2 - \frac{\tau (\gamma^{-d/2} r + \eta^{-d/2} s)}{n^{1/2}} \right ) \\
+ \; &\frac{2 (1 + \nu) \tau \gamma^{-d/2} r}{n^{1/2}} - \frac{2 (1 + \nu) \tau \hat \gamma^{-d/2} \hat r}{n^{1/2}} + \frac{v^{d/2} \tau (2 \hat \gamma^{-d/2} \hat r + \gamma^{-d/2} r)}{u^{d/2} n^{1/2}} \\
\leq \; &40 I_\infty (g, \gamma, r) + \frac{v^{d/2} (3 + 2 \nu) \tau \gamma^{-d/2} r}{u^{d/2} n^{1/2}} + \frac{2 v^{d/2} \tau \hat \gamma^{-d/2} \hat r}{u^{d/2} n^{1/2}}.
\end{align*}
This shows
\begin{equation*}
\lVert V \hat h_{\hat \gamma, \hat r} - V \hat h_{\hat \gamma \wedge \gamma, \hat r \vee r} \rVert_{L_2 (P_n)}^2 \leq 40 I_\infty (g, \gamma, r) + \frac{v^{d/2} (3 + 2 \nu) \tau \gamma^{-d/2} r}{u^{d/2} n^{1/2}} + \frac{2 v^{d/2} \tau \hat \gamma^{-d/2} \hat r}{u^{d/2} n^{1/2}},
\end{equation*}
and it follows from Lemma 31, our choice of $t$ and $\Gamma \subs [u, v]$ that
\begin{align*}
&\lVert V \hat h_{\hat \gamma, \hat r} - V \hat h_{\hat \gamma \wedge \gamma, \hat r \vee r} \rVert_{L_2 (P)}^2 \\
\leq \; &40 I_\infty (g, \gamma, r) + \frac{v^{d/2} (3 + 2 \nu) \tau \gamma^{-d/2} r}{u^{d/2} n^{1/2}} + \frac{2 v^{d/2} \tau \hat \gamma^{-d/2} \hat r}{u^{d/2} n^{1/2}} \\
+ \; &\left ( \frac{151 C \tau}{84 \sigma n^{1/2}} + \frac{C \tau^2}{2646 J^2 \sigma^2 n} \right ) (\hat \gamma \wedge \gamma)^{-d/2} (\hat r \vee r) \\
\leq \; &40 I_\infty (g, \gamma, r) + \frac{v^{d/2} (3 + 2 \nu) \tau \gamma^{-d/2} r}{u^{d/2} n^{1/2}} + \frac{2 v^{d/2} \tau \hat \gamma^{-d/2} \hat r}{u^{d/2} n^{1/2}} \\
+ \; &\left ( \frac{151 C \tau}{84 \sigma n^{1/2}} + \frac{C \tau^2}{2646 J^2 \sigma^2 n} \right ) (v/u)^{d/2} (\hat \gamma^{-d/2} \hat r + \gamma^{-d/2} r) \\
= \; &40 I_\infty (g, \gamma, r) + \frac{v^{d/2} (3 + 2 \nu) \tau \gamma^{-d/2} r}{u^{d/2} n^{1/2}} + \frac{151 C v^{d/2} \tau \gamma^{-d/2} r}{84 u^{d/2} \sigma n^{1/2}} + \frac{C v^{d/2} \tau^2 \gamma^{-d/2} r}{2646 J^2 u^{d/2} \sigma^2 n} \\
+ \; &\left ( \frac{v^{d/2}}{u^{d/2} \nu} + \frac{151 C v^{d/2}}{168 u^{d/2} \sigma \nu} + \frac{C v^{d/2} \tau}{5292 J^2 u^{d/2} \sigma^2 \nu n^{1/2}} \right ) \frac{2 \nu \tau \hat \gamma^{-d/2} \hat r}{n^{1/2}}.
\end{align*}
By \eqref{eGaussCrit}, the definition of $(\hat \gamma, \hat r)$ in \eqref{eGaussEst} and \eqref{eGaussComp}, we have
\begin{align}
&\frac{2 \nu \tau \hat \gamma^{-d/2} \hat r}{n^{1/2}} \nonumber \\
\leq \; &\sup_{\eta \in \Gamma, \eta \leq \hat \gamma} \sup_{s \in R, s \geq \hat r} \left ( \lVert \hat h_{\hat \gamma, \hat r} - \hat h_{\eta, s} \rVert_{L_2 (P_n)}^2 - \frac{\tau (\hat \gamma^{-d/2} \hat r + \eta^{-d/2} s)}{n^{1/2}} \right ) + \frac{2 (1 + \nu) \tau \hat \gamma^{-d/2} \hat r}{n^{1/2}} \nonumber \\
\leq \; &\sup_{\eta \in \Gamma, \eta \leq \gamma} \sup_{s \in R, s \geq r} \left ( \lVert \hat h_{\gamma, r} - \hat h_{\eta, s} \rVert_{L_2 (P_n)}^2 - \frac{\tau (\gamma^{-d/2} r + \eta^{-d/2} s)}{n^{1/2}} \right ) + \frac{2 (1 + \nu) \tau \gamma^{-d/2} r}{n^{1/2}} \nonumber \\
\leq \; &40 I_\infty (g, \gamma, r) + \frac{2 (1 + \nu) \tau \gamma^{-d/2} r}{n^{1/2}}. \label{eGaussCritBd}
\end{align}
Hence,
\begin{align*}
&\lVert V \hat h_{\hat \gamma, \hat r} - V \hat h_{\hat \gamma \wedge \gamma, \hat r \vee r} \rVert_{L_2 (P)}^2 \\
\leq \; &40 I_\infty (g, \gamma, r) + \frac{v^{d/2} (3 + 2 \nu) \tau \gamma^{-d/2} r}{u^{d/2} n^{1/2}} + \frac{151 C v^{d/2} \tau \gamma^{-d/2} r}{84 u^{d/2} \sigma n^{1/2}} + \frac{C v^{d/2} \tau^2 \gamma^{-d/2} r}{2646 J^2 u^{d/2} \sigma^2 n} \\
+ \; &\left ( \frac{2 v^{d/2}}{u^{d/2} \nu} + \frac{151 C v^{d/2}}{84 u^{d/2} \sigma \nu} + \frac{C v^{d/2} \tau}{2646 J^2 u^{d/2} \sigma^2 \nu n^{1/2}} \right ) \left ( 20 I_\infty (g, \gamma, r) + \frac{(1 + \nu) \tau \gamma^{-d/2} r}{n^{1/2}} \right ).
\end{align*}

Since \eqref{eGaussComp} holds simultaneously for all $\gamma, \eta \in \Gamma$ and all $r, s \in R$ such that $\eta \leq \gamma$ and $s \geq r$, we have
\begin{equation*}
\lVert \hat h_{\hat \gamma \wedge \gamma, \hat r \vee r} - \hat h_{\gamma, r} \rVert_{L_2 (P_n)}^2 \leq 40 I_\infty (g, \gamma, r) + \frac{\tau (\gamma^{-d/2} r + (\hat \gamma \wedge \gamma)^{-d/2} (\hat r \vee r))}{n^{1/2}}.
\end{equation*}
This shows
\begin{equation*}
\lVert V \hat h_{\hat \gamma \wedge \gamma, \hat r \vee r} - V \hat h_{\gamma, r} \rVert_{L_2 (P_n)}^2 \leq 40 I_\infty (g, \gamma, r) + \frac{\tau ( \gamma^{-d/2} r + (\hat \gamma \wedge \gamma)^{-d/2} (\hat r \vee r))}{n^{1/2}},
\end{equation*}
and it follows from Lemma 31, our choice of $t$ and \eqref{eGaussCritBd} that
\begin{align*}
&\lVert V \hat h_{\hat \gamma \wedge \gamma, \hat r \vee r} - V \hat h_{\gamma, r} \rVert_{L_2 (P)}^2 \\
\leq \; &40 I_\infty (g, \gamma, r) + \frac{\tau ( \gamma^{-d/2} r + (\hat \gamma \wedge \gamma)^{-d/2} (\hat r \vee r))}{n^{1/2}} \\
+ \; &\left ( \frac{151 C \tau}{84 \sigma n^{1/2}} + \frac{C \tau^2}{2646 J^2 \sigma^2 n} \right ) (\hat \gamma \wedge \gamma)^{-d/2} (\hat r \vee r) \\
= \; &40 I_\infty (g, \gamma, r) + \frac{\tau \gamma^{-d/2} r}{n^{1/2}} + \left ( \frac{\tau}{n^{1/2}} + \frac{151 C \tau}{84 \sigma n^{1/2}} + \frac{C \tau^2}{2646 J^2 \sigma^2 n} \right ) (\hat \gamma \wedge \gamma)^{-d/2} (\hat r \vee r) \\
\leq \; &40 I_\infty (g, \gamma, r) + \frac{\tau \gamma^{-d/2} r}{n^{1/2}} \\
+ \; &\left ( \frac{\tau}{n^{1/2}} + \frac{151 C \tau}{84 \sigma n^{1/2}} + \frac{C \tau^2}{2646 J^2 \sigma^2 n} \right ) (v/u)^{d/2} (\hat \gamma^{-d/2} \hat r + \gamma^{-d/2} r) \\
\leq \; &40 I_\infty (g, \gamma, r) + \frac{2 v^{d/2} \tau \gamma^{-d/2} r}{u^{d/2} n^{1/2}} + \frac{151 C v^{d/2} \tau \gamma^{-d/2} r}{84 u^{d/2} \sigma n^{1/2}} + \frac{C v^{d/2} \tau^2 \gamma^{-d/2} r}{2646 J^2 u^{d/2} \sigma^2 n} \\
+ \; &\left ( \frac{v^{d/2}}{2 u^{d/2} \nu} + \frac{151 C v^{d/2}}{168 u^{d/2} \sigma \nu} + \frac{C v^{d/2} \tau}{5292 J^2 u^{d/2} \sigma^2 \nu n^{1/2}} \right ) \frac{2 \nu \tau \hat \gamma^{-d/2} \hat r}{n^{1/2}} \\
\leq \; &40 I_\infty (g, \gamma, r) + \frac{2 v^{d/2} \tau \gamma^{-d/2} r}{u^{d/2} n^{1/2}} + \frac{151 C v^{d/2} \tau \gamma^{-d/2} r}{84 u^{d/2} \sigma n^{1/2}} + \frac{C v^{d/2} \tau^2 \gamma^{-d/2} r}{2646 J^2 u^{d/2} \sigma^2 n} \\
+ \; &\left ( \frac{v^{d/2}}{u^{d/2} \nu} + \frac{151 C v^{d/2}}{84 u^{d/2} \sigma \nu} + \frac{C v^{d/2} \tau}{2646 J^2 u^{d/2} \sigma^2 \nu n^{1/2}} \right ) \left ( 20 I_\infty (g, \gamma, r) + \frac{(1 + \nu) \tau \gamma^{-d/2} r}{n^{1/2}} \right ).
\end{align*}
Hence,
\begin{align*}
&\lVert V \hat h_{\hat \gamma, \hat r} - V \hat h_{\gamma, r} \rVert_{L_2 (P)}^2 \\
\leq \; &2 \lVert V \hat h_{\hat \gamma, \hat r} - V \hat h_{\hat \gamma \wedge \gamma, \hat r \vee r} \rVert_{L_2 (P)}^2 + 2 \lVert V h_{\hat \gamma \wedge \gamma, \hat r \vee r} - V \hat h_{\gamma, r} \rVert_{L_2 (P)}^2 \\
\leq \; &160 I_\infty (g, \gamma, r) + \frac{2 v^{d/2} (5 + 2 \nu) \tau \gamma^{-d/2} r}{u^{d/2} n^{1/2}} + \frac{151 C v^{d/2} \tau \gamma^{-d/2} r}{21 u^{d/2} \sigma n^{1/2}} + \frac{2 C v^{d/2} \tau^2 \gamma^{-d/2} r}{1323 J^2 u^{d/2} \sigma^2 n} \\
+ \; &\left ( \frac{6 v^{d/2}}{u^{d/2} \nu} + \frac{151 C v^{d/2}}{21 u^{d/2} \sigma \nu} + \frac{2 C v^{d/2} \tau}{1323 J^2 u^{d/2} \sigma^2 \nu n^{1/2}} \right ) \left ( 20 I_\infty (g, \gamma, r) + \frac{(1 + \nu) \tau \gamma^{-d/2} r}{n^{1/2}} \right ).
\end{align*}
We have
\begin{equation*}
\lVert V \hat h_{\hat \gamma, \hat r} - g \rVert_{L_2 (P)}^2 \leq 2 \lVert V \hat h_{\hat \gamma, \hat r} - V \hat h_{\gamma, r} \rVert_{L_2 (P)}^2 + 2 \lVert V \hat h_{\gamma, r} - g \rVert_{L_2 (P)}^2
\end{equation*}
and the result follows. 
\end{proof}

We assume \ref{g1} to bound the distance between $V \hat h_{\hat \gamma, \hat r}$ and $g$ in the $L^2(P)$ norm and prove Theorem \ref{tGaussLepBound}.

{\bf Proof of Theorem \ref{tGaussLepBound}}
By Theorem \ref{tGaussLep}, we have
\begin{align*}
&\lVert V \hat h_{\hat \gamma, \hat r} - g \rVert_{L^2 (P)}^2 \\
\leq \; &\inf_{\gamma \in \Gamma} \inf_{r \in R} \left ( (1 + D_4 \tau n^{-1/2}) (D_5 \tau \gamma^{-d/2} r n^{-1/2} + D_6 I_\infty (g, \gamma, r)) + 2 \lVert V \hat h_{\gamma, r} - g \rVert_{L^2 (P)}^2 \right )
\end{align*}
for some constants $D_4, D_5, D_6 > 0$ not depending on $\tau$, $\gamma$, $r$ or $n$. By Theorem \ref{tVaryBound}, we have
\begin{align*}
\lVert V \hat h_{\gamma, r} - g \rVert_{L^2 (P)}^2 &\leq \frac{(151 C + 21 \sigma) \tau \gamma^{-d/2} r}{42 \sigma n^{1/2}} + \frac{C \tau^2 \gamma^{-d/2} r}{1323 J^2 \sigma^2 n} + 10 I_\infty (g, \gamma, r) \\
&\leq D_7 \tau \gamma^{-d/2} r n^{-1/2} + D_8 \tau^2 \gamma^{-d/2} r n^{-1} + 10 I_\infty (g, \gamma, r)
\end{align*}
for all $\gamma \in \Gamma$ and all $r \in R$, for some constants $D_7, D_8 > 0$ not depending on $\tau$, $\gamma$, $r$ or $n$. This gives
\begin{align*}
\lVert V \hat h_{\hat \gamma, \hat r} - g \rVert_{L^2 (P)}^2 \leq \inf_{\gamma \in \Gamma} \inf_{r \in R} &\left ( (1 + D_4 \tau n^{-1/2}) (D_5 \tau \gamma^{-d/2} r n^{-1/2} + D_6 I_\infty (g, \gamma, r)) \right . \\
&+ \left . 2 D_7 \tau \gamma^{-d/2} r n^{-1/2} + 2 D_8 \tau^2 \gamma^{-d/2} r n^{-1} + 20 I_\infty (g, \gamma, r) \right ).
\end{align*}
Hence, the result follows with
\begin{equation*}
D_1 = \frac{D_4 D_5 + 2 D_8}{D_5 + 2 D_7}, \; D_2 = D_5 + 2 D_7, \; D_3 = D_6 + 20.
\end{equation*} \hfill \BlackBox

We assume \ref{g3} to prove Theorem \ref{tInterGaussLepBound}.

{\bf Proof of Corollary \ref{tInterGaussLepBound}}
If we assume \ref{R1} and \ref{Gamma1}, then $\alpha \in \Gamma$ and $r = a n^{(1-\beta)/(2 (1+\beta))} \in R$, so
\begin{align*}
\lVert V \hat h_{\hat \gamma, \hat r} - g \rVert_{L^2 (P)}^2 &\leq (1 + D_3 \tau n^{-1/2}) (D_4 \tau \alpha^{-d/2} r n^{-1/2} + D_5 I_\infty (g, \alpha, r)) \\
&\leq (1 + D_3 \tau n^{-1/2}) \left ( D_4 \tau \alpha^{-d/2} a n^{- \beta/(1+\beta)} + \frac{D_5 B^{2/(1-\beta)}}{a^{2 \beta/(1-\beta)} n^{\beta/(1+\beta)}} \right )
\end{align*}
for some constants $D_3, D_4, D_5 > 0$ not depending on $n$ or $\tau$ by Theorem \ref{tGaussLepBound} and \eqref{eGaussApproxInter}. If we assume \ref{R2} and \ref{Gamma2}, then there is a unique $\gamma \in \Gamma$ such that $\alpha/c < \gamma \leq \alpha$ and a unique $r \in R$ such that
\begin{equation*}
a n^{(1-\beta)/(2 (1+\beta))} \leq r < a n^{(1-\beta)/(2 (1+\beta))} + b.
\end{equation*}
By Theorem \ref{tGaussLepBound}, Lemma \ref{lNest} and \eqref{eGaussApproxInter}, we have
\begin{align*}
&\lVert V \hat h_{\hat \gamma, \hat r} - g \rVert_{L^2 (P)}^2 \\
\leq \; &(1 + D_3 \tau n^{-1/2}) (D_4 \tau \gamma^{-d/2} r n^{-1/2} + D_5 I_\infty (g, \gamma, r)) \\
\leq \; &(1 + D_3 \tau n^{-1/2}) \left ( D_4 \tau c^{d/2} \alpha^{-d/2} (a n^{(1-\beta)/(2 (1+\beta))} + b) n^{- 1/2} + \frac{D_5 B^{2/(1-\beta)}}{a^{2 \beta/(1-\beta)} n^{\beta/(1+\beta)}} \right ).
\end{align*}
In either case,
\begin{equation*}
\lVert V \hat h_{\hat \gamma, \hat r} -  g \rVert_{L^2 (P)}^2 \leq D_1 \tau n^{-\beta/(1+\beta)} + D_2 \tau^2 n^{-(1+3\beta)/(2(1+\beta))}
\end{equation*}
for some constants $D_1, D_2 > 0$ not depending on $n$ or $\tau$. \hfill \BlackBox

%{\bf Proof of Corollary \ref{InterGaussLepExpBound}} The proof is identical to the proof of Corollary  \ref{InterLepExpBound}.  \hfill \BlackBox
%\begin{supplement}
%\stitle{Regression Results}
%\sdatatype{.pdf}
%\sdescription{We provide the proofs of the results for the two regression problems, including the majorants, along with some technical results.}
%\end{supplement}

\bibliographystyle{plainnat}
\bibliography{paper}

\end{document}

\begin{frontmatter}
\title{Supplement to ``The Goldenshluger--Lepski Method for Constrained Least-Squares Estimators over RKHSs''}
\runtitle{Supplement to Lepski for Constrained Estimators over RKHSs}

\begin{aug}
\author{\fnms{Stephen} \snm{Page} \thanksref{a1,e1} \ead[label=e1,mark]{s.page@lancaster.ac.uk}}
\and
\author{\fnms{Steffen} \snm{Gr\"{u}new\"{a}lder} \thanksref{a2} \ead[label=e2]{s.grunewalder@lancaster.ac.uk}}
\runauthor{Stephen Page and Steffen Gr\"{u}new\"{a}lder}
\affiliation{Lancaster University}
\address[a1]{STOR-i, Lancaster University, Lancaster, LA1 4YF, United Kingdom. \printead{e1}}
\address[a2]{Department of Mathematics and Statistics, Lancaster University, Lancaster, LA1 4YF, United Kingdom. \printead{e2}}
\end{aug}
\end{frontmatter}

In this supplement, we provide the proofs of the results for the two regression problems in \cite{page2018goldenshluger}, including the majorants, along with some technical results.

\appendix

\section{Proof of the Regression Results for a Fixed RKHS}

We bound the distance between $\hat h_r$ and $h_r$ in the $L^2(P_n)$ norm for $r \geq 0$ and $h_r \in r B_H$ to prove the following. 

\begin{lemma} \label{lBias}
Assume (Y) and (H). Let $t \geq 1$ and $A_{1, t} \in \mathcal{F}$ be the set on which
\begin{equation*}
\lVert \hat h_r - h_r \rVert_{L^2 (P_n)}^2 \leq \frac{20 \lVert k \rVert_{\diag}^{1/2} \sigma r t^{1/2}}{n^{1/2}} + 4 \lVert h_r - g \rVert_\infty^2
\end{equation*}
simultaneously for all $r \geq 0$ and all $h_r \in r B_H$. We have $\prob(A_{1, t}) \geq 1 - e^{-t}$.
\end{lemma}
\begin{proof}
The result is trivial for $r = 0$. By Lemma 2 of \cite{page2017ivanov}, we have
\begin{equation*}
\lVert \hat h_r -  h_r \rVert_{L^2 (P_n)}^2 \leq \frac{4}{n} \sum_{i = 1}^n (Y_i - g(X_i)) (\hat h_r (X_i) - h_r (X_i)) + 4 \lVert h_r - g \rVert_{L^2 (P_n)}^2
\end{equation*}
for all $r > 0$ and all $h_r \in r B_H$. We now bound the right-hand side. We have
\begin{equation*}
\lVert h_r - g \rVert_{L^2 (P_n)}^2 \leq \lVert h_r - g \rVert_\infty^2.
\end{equation*}
Furthermore,
\begin{align*}
&\frac{1}{n} \sum_{i = 1}^n (Y_i - g(X_i)) (\hat h_r (X_i) - h_r (X_i)) \\
\leq \; &\sup_{f \in 2 r B_H} \left \lvert \frac{1}{n} \sum_{i = 1}^n (Y_i - g(X_i)) f(X_i) \right \rvert \\
= \; &\sup_{f \in 2 r B_H} \left \lvert \left \langle \frac{1}{n} \sum_{i = 1}^n (Y_i - g(X_i)) k_{X_i}, f \right \rangle_H \right \rvert \\
= \; &2 r \left \lVert \frac{1}{n} \sum_{i = 1}^n (Y_i - g(X_i)) k_{X_i} \right \rVert_H \\
= \; &2 r \left ( \frac{1}{n^2} \sum_{i, j = 1}^n (Y_i - g(X_i)) (Y_j - g(X_j)) k(X_i, X_j) \right )^{1/2}
\end{align*}
by the reproducing kernel property and the Cauchy--Schwarz inequality. Let $K$ be the $n \times n$ matrix with $K_{i, j} = k(X_i, X_j)$ and let $\eps$ be the vector of the $Y_i - g(X_i)$. Then
\begin{equation*}
\frac{1}{n^2} \sum_{i, j = 1}^n (Y_i - g(X_i)) (Y_j - g(X_j)) k(X_i, X_j) = \eps^{\trans} (n^{-2} K) \eps.
\end{equation*}
Furthermore, since $k$ is a measurable function on $(S \times S, \mathcal{S} \otimes \mathcal{S})$, we have that $n^{-2} K$ is an $(\real^{n \times n}, \bor(\real^{n \times n}))$-valued measurable matrix on $(\Omega, \mathcal{F})$ and non-negative-definite. Let $a_i$ for $1 \leq i \leq n$ be the eigenvalues of $n^{-2} K$. Then
\begin{equation*}
\max_{1 \leq i \leq n} a_i \leq \tr(n^{-2} K) \leq n^{-1} \lVert k \rVert_{\diag}
\end{equation*}
and
\begin{equation*}
\tr((n^{-2} K)^2) = \lVert a \rVert_2^2 \leq \lVert a \rVert_1^2 \leq n^{-2} \lVert k \rVert_{\diag}^2.
\end{equation*}
Therefore, by Lemma 36 of \cite{page2017ivanov}, we have
\begin{equation*}
\eps^{\trans} (n^{-2} K) \eps \leq \lVert k \rVert_{\diag} \sigma^2 n^{-1} (1 + 2 t + 2 (t^2 + t)^{1/2})
\end{equation*}
and
\begin{equation*}
\frac{1}{n} \sum_{i = 1}^n (Y_i - g(X_i)) (\hat h_r (X_i) - h_r (X_i)) \leq \frac{5 \lVert k \rVert_{\diag}^{1/2} \sigma r t^{1/2}}{n^{1/2}}
\end{equation*}
with probability at least $1 - e^{-t}$. The result follows. %\hfill \BlackBox
\end{proof}

The following lemma, which is Lemma 25 of \cite{page2017ivanov}, is useful for proving Lemma \ref{lSwapNorm}.

\begin{lemma} \label{lTal}
Let $D > 0$ and $A \subs L^\infty$ be separable with $\lVert f \rVert_\infty \leq D$ for all $f \in A$. Let
\begin{equation*}
Z = \sup_{f \in A} \left \lvert \lVert f \rVert_{L^2 (P_n)}^2 - \lVert f \rVert_{L^2 (P)}^2 \right \rvert.
\end{equation*}
Then, for $t > 0$, we have
\begin{equation*}
Z \leq \expect(Z) + \left ( \frac{2 D^4 t}{n} + \frac{4 D^2 \expect(Z) t}{n} \right )^{1/2} + \frac{2 D^2 t}{3 n}
\end{equation*}
with probability at least $1 - e^{-t}$.
\end{lemma}

We bound the supremum of the difference in the $L^2 (P_n)$ norm and the $L^2 (P)$ norm over $r B_H$ for $r \geq 0$ to prove the next result. We make use of the notion of a \textit{contraction vanishing at $0$}. A function $\varphi:\mathbb{R} \rightarrow \mathbb{R}$ is called a contraction vanishing at $0$ if it is Lipschitz continuous (with Lipschitz constant $1$) and if $\varphi(0) =0 $.

\begin{lemma} \label{lSwapNorm}
Assume (H). Let $t \geq 1$ and $A_{2, t} \in \mathcal{F}$ be the set on which
\begin{equation*}
\sup_{f_1, f_2 \in r B_H} \left \lvert \lVert V f_1 - V f_2 \rVert_{L^2 (P_n)}^2 - \lVert V f_1 - V f_2 \rVert_{L^2 (P)}^2 \right \rvert \leq \frac{97 \lVert k \rVert_{\diag}^{1/2} C r t^{1/2}}{n^{1/2}} + \frac{8 \lVert k \rVert_{\diag}^{1/2} C r t}{3 n}
\end{equation*}
simultaneously for all $r \geq 0$. We have $\prob(A_{2, t}) \geq 1 - e^{-t}$.
\end{lemma}
\begin{proof}
The result is trivial for $r = 0$. Let
\begin{equation*}
Z = \sup_{r > 0} \sup_{f_1, f_2 \in r B_H} \frac{1}{r} \left \lvert \lVert V f_1 - V f_2 \rVert_{L^2 (P_n)}^2 - \lVert V f_1 - V f_2 \rVert_{L^2 (P)}^2 \right \rvert.
\end{equation*}
Furthermore, let the $\eps_i$ for $1 \leq i \leq n$ be \iid Rademacher random variables on $(\Omega, \mathcal{F}, \prob)$, independent of the $X_i$. Lemma 2.3.1 of \cite{van1996weak} shows
\begin{equation*}
\expect(Z) \leq 2 \expect \left ( \sup_{r > 0} \sup_{f_1, f_2 \in r B_H} \left \lvert \frac{1}{n} \sum_{i = 1}^n \eps_i (r^{-1/2} V f_1 (X_i) - r^{-1/2} V f_2 (X_i))^2 \right \rvert \right )
\end{equation*}
by symmetrisation. Since
\begin{equation*}
\lvert V f_1 (X_i) - V f_2 (X_i) \rvert \leq 2 C
\end{equation*}
for all $r > 0$ and all $f_1, f_2 \in r B_H$, we find
\begin{equation*}
\frac{(r^{-1/2} V f_1 (X_i) - r^{-1/2} V f_2 (X_i))^2}{4 C}
\end{equation*}
is a contraction vanishing at 0 as a function of $r^{-1} V f_1 (X_i) - r^{-1} V f_2 (X_i)$ for all $1 \leq i \leq n$. By Theorem 3.2.1 of \cite{gine2015mathematical}, we have
\begin{equation*}
\expect \left (\sup_{r > 0} \sup_{f_1, f_2 \in r B_H} \left \lvert \frac{1}{n} \sum_{i = 1}^n \eps_i \frac{(r^{-1/2} V f_1 (X_i) - r^{-1/2} V f_2 (X_i))^2}{4 C} \right \rvert \; \middle \vert X \right )
\end{equation*}
is at most
\begin{equation*}
2 \expect \left (\sup_{r > 0} \sup_{f_1, f_2 \in r B_H} \left \lvert \frac{1}{n} \sum_{i = 1}^n \eps_i (r^{-1} V f_1 (X_i) - r^{-1} V f_2 (X_i)) \right \rvert \; \middle \vert X \right )
\end{equation*}
almost surely. Therefore,
\begin{align*}
\expect(Z) &\leq16 C \expect \left (\sup_{r > 0} \sup_{f_1, f_2 \in r B_H} \left \lvert \frac{1}{n} \sum_{i = 1}^n \eps_i (r^{-1} V f_1 (X_i) - r^{-1} V f_2 (X_i)) \right \rvert \right ) \\
\leq \; &32 C \expect \left (\sup_{r > 0} \sup_{f \in r B_H} \left \lvert \frac{1}{n} \sum_{i = 1}^n \eps_i r^{-1} V f(X_i) \right \rvert \right )
\end{align*}
by the triangle inequality. Again, by Theorem 3.2.1 of \cite{gine2015mathematical}, we have
\begin{equation*}
\expect(Z) \leq 64 C \expect \left (\sup_{r > 0} \sup_{f \in r B_H} \left \lvert \frac{1}{n} \sum_{i = 1}^n \eps_i r^{-1} f(X_i) \right \rvert \right )
\end{equation*}
since $V$ is a contraction vanishing at 0. We have
\begin{align*}
\sup_{r > 0} \sup_{f \in r B_H} \left \lvert \frac{1}{n} \sum_{i = 1}^n \eps_i r^{-1} f(X_i) \right \rvert &= \sup_{r > 0} \sup_{f \in r B_H} \left \lvert \left \langle \frac{1}{n} \sum_{i = 1}^n \eps_i k_{X_i}, r^{-1} f \right \rangle_H \right \rvert \\
&= \left \lVert \frac{1}{n} \sum_{i = 1}^n \eps_i k_{X_i} \right \rVert_H \\
&= \left ( \frac{1}{n^2} \sum_{i, j = 1}^n \eps_i \eps_j k(X_i, X_j) \right )^{1/2}.
\end{align*}
by the reproducing kernel property and the Cauchy--Schwarz inequality. By Jensen's inequality, we have
\begin{align*}
\expect \left ( \sup_{r > 0} \sup_{f \in r B_H} \left \lvert \frac{1}{n} \sum_{i = 1}^n \eps_i r^{-1} f(X_i) \right \rvert \; \middle \vert X \right ) &\leq \left ( \frac{1}{n^2} \sum_{i, j = 1}^n \cov(\eps_i, \eps_j \vert X ) k(X_i, X_j) \right )^{1/2} \\
&= \left ( \frac{1}{n^2} \sum_{i = 1}^n k(X_i, X_i) \right )^{1/2}
\end{align*}
almost surely and again, by Jensen's inequality, we have
\begin{equation*}
\expect \left ( \sup_{r > 0} \sup_{f \in r B_H} \left \lvert \frac{1}{n} \sum_{i = 1}^n \eps_i r^{-1} f(X_i) \right \rvert \right ) \leq \left ( \frac{\lVert k \rVert_{\diag}}{n} \right )^{1/2}.
\end{equation*}
Hence, $\expect(Z) \leq 64 \lVert k \rVert_{\diag}^{1/2} C n^{-1/2}$.

Let
\begin{equation*}
A = \left \{ r^{-1/2} V f_1 - r^{-1/2} V f_2 : r > 0 \text{ and } f_1, f_2 \in r B_H \right \}.
\end{equation*}
We have that $(0, \infty)$, the set indexing $r$, is separable. Furthermore, $H$ is separable and so is separable in $C(S)$ as it can be continuously embedded in $C(S)$ due to its bounded kernel. Therefore, $r B_H \subs H$ is separable in $C(S)$ for $r > 0$. Hence, we have that $A \subs C(S)$ is separable. Furthermore,
\begin{align*}
\left \lVert r^{-1/2} V f_1 - r^{-1/2} V f_2 \right \rVert_\infty &\leq \min \left ( 2 C r^{-1/2}, 2 \lVert k \rVert_{\diag}^{1/2} r^{1/2} \right ) \\
&\leq 2 \lVert k \rVert_{\diag}^{1/4} C^{1/2}
\end{align*}
for all $r > 0$ and all $f_1, f_2 \in r B_H$. The first term in the minimum comes from clipping using $V$, while the second term comes from the continuous embedding of $H$ in $C(S)$ due to its bounded kernel. By Lemma \ref{lTal}, we have
\begin{equation*}
Z \leq \expect(Z) + \left ( \frac{32 \lVert k \rVert_{\diag} C^2 t}{n} + \frac{16 \lVert k \rVert_{\diag}^{1/2} C \expect(Z) t}{n} \right )^{1/2} + \frac{8 \lVert k \rVert_{\diag}^{1/2} C t}{3 n}
\end{equation*}
with probability at least $1 - e^{-t}$. We have $\expect(Z) \leq 64 \lVert k \rVert_{\diag}^{1/2} C n^{-1/2}$ from above. The result follows. %\hfill \BlackBox
\end{proof}

We move the bound on the distance between $V \hat h_r$ and $V h_r$ from the $L^2(P_n)$ norm to the $L^2 (P)$ norm for $r \geq 0$ and $h_r \in r B_H$.

\begin{corollary} \label{cTrueNorm}
Assume (Y) and (H). Let $t \geq 1$ and recall the definitions of $A_{1, t}$ and $A_{2, t}$ from Lemmas \ref{lBias} and \ref{lSwapNorm}. On the set $A_{1, t} \cap A_{2, t} \in \mathcal{F}$, for which $\prob(A_{1, t} \cap A_{2, t}) \geq 1 - 2 e^{-t}$, we have
\begin{equation*}
\lVert V \hat h_r - V h_r \rVert_{L^2 (P)}^2 \leq \frac{\lVert k \rVert_{\diag}^{1/2}(97 C + 20 \sigma)r t^{1/2}}{n^{1/2}} + \frac{8 \lVert k \rVert_{\diag}^{1/2} C r t}{3 n} + 4 \lVert h_r - g \rVert_\infty^2
\end{equation*}
simultaneously for all $r \geq 0$ and all $h_r \in r B_H$.
\end{corollary}

\begin{proof}
By Lemma \ref{lBias}, we have
\begin{equation*}
\lVert \hat h_r - h_r \rVert_{L^2 (P_n)}^2 \leq \frac{20 \lVert k \rVert_{\diag}^{1/2} \sigma r t^{1/2}}{n^{1/2}} + 4 \lVert h_r - g \rVert_\infty^2
\end{equation*}
for all $r \geq 0$ and all $h_r \in r B_H$, so
\begin{equation*}
\lVert V \hat h_r -  V h_r \rVert_{L^2 (P_n)}^2 \leq \frac{20 \lVert k \rVert_{\diag}^{1/2} \sigma r t^{1/2}}{n^{1/2}} + 4 \lVert h_r - g \rVert_\infty^2.
\end{equation*}
Since $\hat h_r, h_r \in r B_H$, by Lemma \ref{lSwapNorm} we have
\begin{align*}
&\lVert V \hat h_r -  V h_r \rVert_{L^2 (P)}^2 - \lVert V \hat h_r - V h_r \rVert_{L^2 (P_n)}^2 \\
\leq \; &\sup_{f_1, f_2 \in r B_H} \left \lvert \lVert V f_1 - V f_2 \rVert_{L^2 (P_n)}^2 - \lVert V f_1 - V f_2 \rVert_{L^2 (P)}^2 \right \rvert \\
\leq \; &\frac{97 \lVert k \rVert_{\diag}^{1/2} C r t^{1/2}}{n^{1/2}} + \frac{8 \lVert k \rVert_{\diag}^{1/2} C r t}{3 n}.
\end{align*}
The result follows.
\end{proof}

We assume (g1) to bound the distance between $V \hat h_r$ and $g$ in the $L^2(P)$ norm for $r \geq 0$ and prove Theorem 1.

{\bf Proof of Theorem 1}
Note that $V g = g$. We have
\begin{align*}
\lVert V \hat h_r - g \rVert_{L^2 (P)}^2 &\leq \left ( \lVert V \hat h_r - V h_r \rVert_{L^2 (P)} + \lVert V h_r - g \rVert_{L^2 (P)} \right )^2 \\
&\leq 2 \lVert V \hat h_r - V h_r \rVert_{L^2 (P)}^2 + 2 \lVert V h_r - g \rVert_{L^2 (P)}^2 \\
&\leq 2 \lVert V \hat h_r - V h_r \rVert_{L^2 (P)}^2 + 2 \lVert h_r - g \rVert_{L^2 (P)}^2
\end{align*}
for all $r \geq 0$ and all $h_r \in r B_H$. By Corollary \ref{cTrueNorm}, we have
\begin{equation*}
\lVert V \hat h_r - V h_r \rVert_{L^2 (P)}^2 \leq \frac{\lVert k \rVert_{\diag}^{1/2}(97 C + 20 \sigma)r t^{1/2}}{n^{1/2}} + \frac{8 \lVert k \rVert_{\diag}^{1/2} C r t}{3 n} + 4 \lVert h_r - g \rVert_\infty^2.
\end{equation*}
Hence,
\begin{equation*}
\lVert V \hat h_r - g \rVert_{L^2 (P)}^2 \leq \frac{2 \lVert k \rVert_{\diag}^{1/2}(97 C + 20 \sigma)r t^{1/2}}{n^{1/2}} + \frac{16 \lVert k \rVert_{\diag}^{1/2} C r t}{3 n} + 10 \lVert h_r - g \rVert_\infty^2.
\end{equation*}
Taking an infimum over $h_r \in r B_H$ proves the result. \hfill \BlackBox

\section{Proof of the Regression Results for a Collection of RKHSs}

\begin{lemma} \label{lVaryEst}
Assume ($\mathcal{K}1$). We have that $\hat h_{k, r}$ is an $(C(S), \bor(C(S)))$-valued measurable function on $(\Omega \times \mathcal{K} \times [0, \infty), \mathcal{F} \otimes \bor(\mathcal{K}) \otimes \bor([0, \infty)))$, where $k$ varies in $\mathcal{K}$ and $r$ varies in $[0, \infty)$.
\end{lemma}
\begin{proof}
Let $K$ be the $n \times n$ symmetric matrix with $K_{i, j} = k(X_i, X_j)$ for $k \in \mathcal{K}$. Then $K$ is a continuous function of $k$ and $X$, hence it is an $(\real^{n \times n}, \bor(\real^{n \times n}))$-valued measurable matrix on $(\Omega \times \mathcal{K}, \mathcal{F} \otimes \bor(\mathcal{K}))$, where $k$ varies in $\mathcal{K}$. By Lemma \ref{lDiag}, there exist an orthogonal matrix $A$ and a diagonal matrix $D$ which are both $(\real^{n \times n}, \bor(\real^{n \times n}))$-valued measurable matrices on $(\Omega \times \mathcal{K}, \mathcal{F} \otimes \bor(\mathcal{K}))$ such that $K = A D A^{\trans}$. Since $K$ is non-negative definite, the diagonal entries of $D$ are non-negative, and we may assume that they are non-increasing. Let $m = \rk K$, which is measurable on $(\Omega \times \mathcal{K}, \mathcal{F} \otimes \bor(\mathcal{K}))$. For $r > 0$, if
\begin{equation*}
r^2 < \sum_{i = 1}^m D_{i, i}^{-1} (A^{\trans} Y)_i^2,
\end{equation*}
then define $\mu(r) > 0$ by
\begin{equation*}
\sum_{i = 1}^m \frac{D_{i, i}}{(D_{i, i} + n \mu(r))^2} (A^{\trans} Y)_i^2 = r^2.
\end{equation*}
Otherwise, let $\mu(r) = 0$. Let $a \in \real^n$ be defined by
\begin{equation*}
(A^{\trans} a)_i = (D_{i, i} + n \mu(r))^{-1} (A^{\trans} Y)_i
\end{equation*}
for $1 \leq i \leq m$ and $(A^{\trans} a)_i = 0$ for $m+1 \leq i \leq n$, noting that $A^{\trans}$ has the inverse $A$ since it is an orthogonal matrix. By Lemma 3 of \cite{page2017ivanov},
\begin{equation*}
\hat h_{k, r} = \sum_{i = 1}^n a_i k_{X_i}
\end{equation*}
for $r > 0$ and $\hat h_{k, 0} = 0$ for $k \in \mathcal{K}$.

Since $\mu(r) > 0$ is strictly decreasing for
\begin{equation*}
r^2 < \sum_{i = 1}^m D_{i, i}^{-1} (A^{\trans} Y)_i^2
\end{equation*}
and $\mu(r) = 0$ otherwise, we find
\begin{equation*}
\{ \mu(r) \leq \mu \} = \left \{ \sum_{i = 1}^m \frac{D_{i, i}}{(D_{i, i} + n \mu)^2} (A^{\trans} Y)_i^2 \leq r^2 \right \}
\end{equation*}
for $\mu \in [0, \infty)$. Therefore, $\mu(r)$ is measurable on $(\Omega \times \mathcal{K} \times [0, \infty), \mathcal{F} \otimes \bor(\mathcal{K}) \otimes \bor([0, \infty)))$, where $k$ varies in $\mathcal{K}$ and $r$ varies in $[0, \infty)$. Hence, the $a$ above with $\mu = \mu(r)$ for $r > 0$ is measurable on $(\Omega \times \mathcal{K} \times [0, \infty), \mathcal{F} \otimes \bor(\mathcal{K}) \otimes \bor([0, \infty)))$. By Lemma 4.29 of \cite{steinwart2008support}, $\Phi_k : S \to H_k$ by $\Phi_k (x) = k_x$ is continuous for all $k \in \mathcal{K}$. Hence, $\Phi : \mathcal{K} \times S \to C(S)$ by $\Phi(k, x) = k_x$ is continuous and $k_{X_i}$ for $1 \leq i \leq n$ are $(C(S), \bor(C(S)))$-valued measurable functions on $(\Omega \times \mathcal{K}, \mathcal{F} \otimes \bor(\mathcal{K}))$. Together, these show that $\hat h_{k, r}$ is an $(C(S), \bor(C(S)))$-valued measurable function on $(\Omega \times \mathcal{K} \times [0, \infty), \mathcal{F} \otimes \bor(\mathcal{K}) \otimes \bor([0, \infty)))$, where $k$ varies in $\mathcal{K}$ and $r$ varies in $[0, \infty)$, recalling that $\hat h_{k, 0}$ = 0.  
\end{proof}

Let $\psi_1 (x) = \exp( \lvert x \rvert ) - 1$ for $x \in \real$ and
\begin{equation*}
\lVert Z \rVert_{\psi_1} = \inf \{ a \in (0, \infty) : \expect(\psi_1 (Z / a)) \leq 1 \}
\end{equation*}
for any random variable $Z$ on $(\Omega, \mathcal{F})$. Note that this infimum is attained by the monotone convergence theorem, and $\lVert Z \rVert_{\psi_1}$ increases as $\lvert Z \rvert$ increases pointwise. Let $L^{\psi_1}$ be the set of random variables $Z$ on $(\Omega, \mathcal{F}, \prob)$ such that $\lVert Z \rVert_{\psi_1} < \infty$. We have that $(L^{\psi_1}, \lVert \cdot \rVert_{\psi_1})$ is a Banach space known as an Orlicz space (see \cite{rao1991theory}).

\begin{lemma} \label{lPsi1}
Let $Z \in L^{\psi_1}$. We have
\begin{equation*}
\expect(\lvert Z \rvert) \leq (\log 2) \lVert Z \rVert_{\psi_1}.
\end{equation*}
Let $t \geq 0$. We have
\begin{equation*}
\lvert Z \rvert \leq \lVert Z \rVert_{\psi_1} (\log 2 + t)
\end{equation*}
with probability at least $1 - e^{-t}$.
\end{lemma}

\begin{proof}
We have $\expect(\exp (\lvert Z \rvert / \lVert Z \rVert_{\psi_1})) \leq 2$. The first result follows from Jensen's inequality. The second result follows from Chernoff bounding.
\end{proof}
For $m \times n$ matrices $U$ and $V$, define $U \circ V$ to be the $m \times n$ matrix with
\begin{equation*}
(U \circ V)_{i, j} = U_{i, j} V_{i, j}.
\end{equation*}
Recall that
\begin{align*}
\mathcal{L} &= \{ k / \lVert k \rVert_{\diag} : k \in \mathcal{K} \} \cup \{0\}, \\
D &= \sup_{f_1, f_2 \in \mathcal{L}} \lVert f_1 - f_2 \rVert_\infty \leq 2, \\
J &= \left ( 162 \int_{0}^{D/2} \log(2 N(a, \mathcal{L}, \lVert \cdot \rVert_\infty)) da + 1 \right )^{1/2}.
\end{align*}
The following lemma is useful for proving Lemma \ref{lVaryBias}.

\begin{lemma} \label{lKernChain}
Assume ($\mathcal{K}1$). Let the $\eps_i$ for $1 \leq i \leq n$ be random variables on $(\Omega, \mathcal{F}, \prob)$ such that $(X_i, \eps_i)$ are \iid  and $\eps_i$ is $\sigma^2$-subgaussian given $X_i$. Let
\begin{equation*}
W(f) = \frac{1}{n^2} \sum_{i, j = 1}^n \eps_i \eps_j f(X_i, X_j)
\end{equation*}
for $f \in \mathcal{L}$. We have
\begin{equation*}
\left \lVert \sup_{f \in \mathcal{L}} W(f) \right \rVert_{\psi_1} \leq \frac{4 J^2 \sigma^2}{n}.
\end{equation*}
\end{lemma}

\begin{proof}
Let $F$ be the $n \times n$ matrix with $F_{i, j} = f(X_i, X_j)$, where $F$ varies with $f \in \mathcal{L}$. Note that $F$ is an $(\real^{n \times n}, \bor(\real^{n \times n}))$-valued measurable matrix on $(\Omega, \mathcal{F})$. Then $W(f) = n^{-2} \eps^{\trans} F \eps$. Let $Z(f) = n^{-2} \eps^{\trans} (F - I \circ F ) \eps$ for $f \in \mathcal{L}$. Note that $Z$ is continuous in $f$. We have
\begin{equation*}
\left \lVert Z(f_1) - Z(f_2) \right \rVert_{\psi_1} \leq 36 \sigma^2 n^{-1} \lVert f_1 - f_2 \rVert_\infty
\end{equation*}
for $f_1, f_2 \in \mathcal{L}$: let $F^{(1)}$ and $F^{(2)}$ be the  matrices corresponding to $f_1$ and $f_2$. Observe that $\left \lVert Z(f_1) - Z(f_2) \right \rVert_{\psi_1} \leq a$ by Lemma \ref{lQuadPsi1} when $a$ is chosen such that
\[
n^2 a = 2^{7/2} (\log 2) \sigma^2 (\tr( F^{(1)} - F^{(2)})^2)^{1/2} /\log(5/4)  \leq 36 \sigma^2 n \|f_1 - f_2\|_\infty.  
\]
 Let $d(f_1, f_2) = 36 \sigma^2 n^{-1} \lVert f_1 - f_2 \rVert_\infty$ for $f_1, f_2 \in \mathcal{L}$ and
\begin{equation*}
D_d = \sup_{f_1, f_2 \in \mathcal{L}} d(f_1, f_2).
\end{equation*}
By Lemma \ref{lChainPsi1} with $M = \mathcal{L}$ and $s_0 = 0$, we find
\begin{align*}
\left \lVert \sup_{f \in \mathcal{L}} \lvert Z(f) \rvert \right \rVert_{\psi_1} &\leq 18 \int_{0}^{D_d/2} \log(2 N(a, \mathcal{L}, d)) da \\
&= \frac{648 \sigma^2}{n} \int_{0}^{D/2} \log(2 N(a, \mathcal{L}, \lVert \cdot \rVert_\infty)) da.
\end{align*}
Hence,
\begin{equation*}
\left \lVert \sup_{f \in \mathcal{L}} W(f) \right \rVert_{\psi_1} \leq \left \lVert n^{-2} \sup_{f \in \mathcal{L}} \eps^{\trans} (I \circ F) \eps \right \rVert_{\psi_1} + \frac{648 \sigma^2}{n} \int_{0}^{D/2} \log(2 N(a, \mathcal{L}, \lVert \cdot \rVert_\infty)) da.
\end{equation*}

We have
\begin{equation*}
n^{-2} \sup_{f \in \mathcal{L}} \eps^{\trans} (I \circ F) \eps \leq n^{-2} \eps^{\trans} \eps,
\end{equation*}
noting that $F_{i, i} \in [0, 1]$ for $1 \leq i \leq n$ and $f \in \mathcal{L}$. Let $\delta_i$ for $1 \leq i \leq n$ be random variables on $(\Omega, \mathcal{F}, \prob)$ which are independent of each other and the $\eps_i$, with $\delta_i \sim \norm(0, \sigma^2)$. Lemma 35 of \cite{page2017ivanov} shows
\begin{align*}
\expect \left ( \exp \left ( n^{-2} t \sup_{f \in \mathcal{L}} \eps^{\trans} (I \circ F) \eps \right ) \right ) &\leq \expect \left ( \exp \left ( n^{-2} t \eps^{\trans} \eps \right ) \right ) \\
&\leq \expect \left ( \exp \left ( n^{-2} t \delta^{\trans} \delta \right ) \right ) \\
&= \prod_{i = 1}^n \left (1 - 2 \sigma^2 n^{-2} t \right )^{-1/2}
\end{align*}
for $0 \leq 2 \sigma^2 n^{-2} t < 1$ by computing the moment generating function of the $\delta_i^2$. We have that $(1 - x)^{- 1/2} \leq \exp(x)$ for $x \in [0, 1/2]$, so
\begin{equation*}
\expect \left ( \exp \left ( n^{-2} t \sup_{f \in \mathcal{L}} \eps^{\trans} (I \circ F) \eps \right ) \right ) \leq \prod_{i = 1}^n \exp \left ( 2 \sigma^2 n^{-2} t \right ) = \exp \left ( 2 \sigma^2 n^{-1} t \right )
\end{equation*}
for $0 \leq 4 \sigma^2 n^{-2} t \leq 1$. This bound is at most 2 and valid for
\begin{equation*}
t \leq \min \left ( \frac{n^2}{4 \sigma^2}, \frac{(\log 2) n}{2 \sigma^2} \right ).
\end{equation*}
Hence,
\begin{equation*}
\left \lVert n^{-2} \sup_{f \in \mathcal{L}} \eps^{\trans} (I \circ F) \eps \right \rVert_{\psi_1} \leq \max \left ( \frac{4 \sigma^2}{n^2}, \frac{2 \sigma^2}{(\log 2) n} \right ) \leq \frac{4 \sigma^2}{n}
\end{equation*}
and
\begin{equation*}
\left \lVert \sup_{f \in \mathcal{L}} W(f) \right \rVert_{\psi_1} \leq \frac{648 \sigma^2}{n} \int_{0}^{D/2} \log(2 N(a, \mathcal{L}, \lVert \cdot \rVert_\infty)) da + \frac{4 \sigma^2}{n}.
\end{equation*}
The result follows.
\end{proof}
We bound the distance between $\hat h_{k, r}$ and $h_{k, r}$ in the $L^2(P_n)$ norm for $k \in \mathcal{K}$, $r \geq 0$ and $h_{k, r} \in r B_k$ to prove the following Lemma.

\begin{lemma} \label{lVaryBias}
Assume (Y) and ($\mathcal{K}1$). Let $t \geq 1$. There exists a set $A_{3, t} \in \mathcal{F}$ with $\prob(A_{3, t}) \geq 1 - e^{-t}$ on which
\begin{equation*}
\lVert \hat h_{k, r} - h_{k, r} \rVert_{L^2 (P_n)}^2 \leq \frac{21 J \lVert k \rVert_{\diag}^{1/2} \sigma r t^{1/2}}{n^{1/2}} + 4 \lVert h_{k, r} - g \rVert_\infty^2
\end{equation*}
simultaneously for all $k \in \mathcal{K}$, all $r \geq 0$ and all $h_{k, r} \in r B_k$.
\end{lemma}
\begin{proof}
The result is trivial for $r = 0$. By Lemma 2 of \cite{page2017ivanov}, we have
\begin{equation*}
\lVert \hat h_{k, r} -  h_{k, r} \rVert_{L^2 (P_n)}^2 \leq \frac{4}{n} \sum_{i = 1}^n (Y_i - g(X_i)) (\hat h_{k, r} (X_i) - h_{k, r} (X_i)) + 4 \lVert h_{k, r} - g \rVert_{L^2 (P_n)}^2
\end{equation*}
for all $k \in \mathcal{K}$, all $r > 0$ and all $h_{k, r} \in r B_k$. We now bound the right-hand side. We have
\begin{equation*}
\lVert h_{k, r} - g \rVert_{L^2 (P_n)}^2 \leq \lVert h_{k, r} - g \rVert_\infty^2.
\end{equation*}
Furthermore,
\begin{align*}
&\frac{1}{n} \sum_{i = 1}^n (Y_i - g(X_i)) (\hat h_{k, r} (X_i) - h_{k, r} (X_i)) \\
\leq \; &\sup_{f \in 2 r B_k} \left \lvert \frac{1}{n} \sum_{i = 1}^n (Y_i - g(X_i)) f(X_i) \right \rvert \\
= \; &\sup_{f \in 2 r B_k} \left \lvert \left \langle \frac{1}{n} \sum_{i = 1}^n (Y_i - g(X_i)) k_{X_i}, f \right \rangle_k \right \rvert \\
= \; &2 r \left \lVert \frac{1}{n} \sum_{i = 1}^n (Y_i - g(X_i)) k_{X_i} \right \rVert_k \\
= \; &2 r \left ( \frac{1}{n^2} \sum_{i, j = 1}^n (Y_i - g(X_i)) (Y_j - g(X_j)) k(X_i, X_j) \right )^{1/2}
\end{align*}
by the reproducing kernel property and the Cauchy--Schwarz inequality. Let
\begin{equation*}
Z = \sup_{k \in \mathcal{K}} \left ( \frac{1}{\lVert k \rVert_{\diag} n^2} \sum_{i, j = 1}^n (Y_i - g(X_i)) (Y_j - g(X_j)) k(X_i, X_j) \right ).
\end{equation*}
By Lemma \ref{lKernChain} with $\eps_i = Y_i - g(X_i)$, we have $\lVert Z \rVert_{\psi_1} \leq 4 J^2 \sigma^2 n^{-1}$. By Lemma \ref{lPsi1}, we have $Z \leq 4 J^2 \sigma^2 (\log 2 + t) n^{-1}$ with probability at least $1 - e^{-t}$. The result follows. 
\end{proof}

The following lemma is useful for proving Lemma \ref{lVarySwapNorm}.

\begin{lemma} \label{lSep}
Let
\begin{equation*}
A = \left \{ \lVert k \rVert_{\diag}^{-1/4} r^{-1/2} V f_1 - \lVert k \rVert_{\diag}^{-1/4} r^{-1/2} V f_2 : k \in \mathcal{K}, r > 0 \text{ and } f_1, f_2 \in r B_k \right \}.
\end{equation*}
Then $A$ is separable as a subset of $C(S)$.
\end{lemma}

\begin{proof}
By Theorem 4.21 of \cite{steinwart2008support}, we have that
\begin{equation*}
\left \{ \sum_{i = 1}^m a_i k_{s_i} : m \geq 1 \text{ and } a_i \in \real, s_i \in S \text{ for } 1 \leq i \leq m \right \}
\end{equation*}
is dense in $H_k$ for $k \in \mathcal{K}$. Hence,
\begin{equation*}
\left \{ \sum_{i = 1}^m a_i k_{s_i} : m \geq 1 \text{ and } a_i \in \real, s_i \in S \text{ for } 1 \leq i \leq m \text{ with} \sum_{i, j = 1}^m a_i a_j k(s_i, s_j) \leq r^2 \right \}
\end{equation*}
is dense in $r B_k \subs H_k$ for $k \in \mathcal{K}$ and $r > 0$. Since $S$ is separable, it has a countable dense subset $S_0$. Let $D_{k, r}$ be
\begin{equation*}
\left \{ \sum_{i = 1}^m a_i k_{s_i} : m \geq 1 \text{ and } a_i \in \rat, s_i \in S_0 \text{ for } 1 \leq i \leq m \text{ with} \sum_{i, j = 1}^m a_i a_j k(s_i, s_j) \leq r^2 \right \}
\end{equation*}
for $k \in \mathcal{K}$ and $r > 0$. Since the function $\Phi_k : S \to H_k$ by $\Phi_k (x) = k_x$ is continuous by Lemma 4.29 of \cite{steinwart2008support}, we have that $D_{k, r}$ is dense in $r B_k \subs H_k$ by suitable choices for $a_i \in \rat$ for $1 \leq i \leq m$. Since $k$ is bounded for all $k \in \mathcal{K}$, as subsets of $C(S)$ we have that $D_{k, r}$ is dense in $r B_k$ and
\begin{equation*}
A = \cl \left ( \left \{ \lVert k \rVert_{\diag}^{-1/4} r^{-1/2} (V f_1 - V f_2) : k \in \mathcal{K}, r > 0 \text{ and } f_1, f_2 \in D_{k, r} \right \} \right ).
\end{equation*}
Since $(\mathcal{K}, \lVert \cdot \rVert_\infty)$ is separable, it has a countable dense subset $\mathcal{K}_0$. Hence,
\begin{equation*}
A = \cl \left ( \left \{ \lVert k \rVert_{\diag}^{-1/4} r^{-1/2} (V f_1 - V f_2) : k \in \mathcal{K}_0, r \in (0, \infty) \cap \rat \text{ and } f_1, f_2 \in D_{k, r} \right \} \right )
\end{equation*}
by suitable choices for $r \in (0, \infty) \cap \rat$. The result follows.
\end{proof}
We bound the supremum of the difference in the $L^2 (P_n)$ norm and the $L^2 (P)$ norm over $r B_k$ for $k \in \mathcal{K}$ and $r \geq 0$ to prove the following Lemma.

\begin{lemma} \label{lVarySwapNorm}
Assume ($\mathcal{K}1$). Let $t \geq 1$ and $A_{4, t} \in \mathcal{F}$ be the set on which
\begin{equation*}
\sup_{f_1, f_2 \in r B_k} \left \lvert \lVert V f_1 - V f_2 \rVert_{L^2 (P_n)}^2 - \lVert V f_1 - V f_2 \rVert_{L^2 (P)}^2 \right \rvert \leq \frac{151 J \lVert k \rVert_{\diag}^{1/2} C r t^{1/2}}{n^{1/2}} + \frac{8 \lVert k \rVert_{\diag}^{1/2} C r t}{3 n}
\end{equation*}
simultaneously for all $k \in \mathcal{K}$ and all $r \geq 0$. We have $\prob(A_{4, t}) \geq 1 - e^{-t}$.
\end{lemma}
\begin{proof}
The result is trivial for $r = 0$. Let
\begin{equation*}
Z = \sup_{k \in \mathcal{K}} \sup_{r > 0} \sup_{f_1, f_2 \in r B_k} \lVert k \rVert_{\diag}^{-1/2} r^{-1} \left \lvert \lVert V f_1 - V f_2 \rVert_{L^2 (P_n)}^2 - \lVert V f_1 - V f_2 \rVert_{L^2 (P)}^2 \right \rvert.
\end{equation*}
We have that $Z$ is a random variable by Lemma \ref{lSep}. Furthermore, let the $\eps_i$ for $1 \leq i \leq n$ be \iid Rademacher random variables on $(\Omega, \mathcal{F}, \prob)$, independent of the $X_i$. Lemma 2.3.1 of \cite{van1996weak} shows
\begin{equation*}
\expect(Z) \leq 2 \expect \left ( \sup_{k \in \mathcal{K}} \sup_{r > 0} \sup_{f_1, f_2 \in r B_k} \lVert k \rVert_{\diag}^{-1/2} \left \lvert \frac{1}{n} \sum_{i = 1}^n \eps_i (r^{-1/2} V f_1 (X_i) - r^{-1/2} V f_2 (X_i))^2 \right \rvert \right )
\end{equation*}
by symmetrisation. Since
\begin{equation*}
\lvert V f_1 (X_i) - V f_2 (X_i) \rvert \leq 2 C
\end{equation*}
for all $k \in \mathcal{K}$, all $r > 0$ and all $f_1, f_2 \in r B_k$, we find
\begin{equation*}
\frac{(r^{-1/2} V f_1 (X_i) - r^{-1/2} V f_2 (X_i))^2}{4 C}
\end{equation*}
is a contraction vanishing at 0 as a function of $r^{-1} V f_1 (X_i) - r^{-1} V f_2 (X_i)$ for all $1 \leq i \leq n$. By Theorem 3.2.1 of \cite{gine2015mathematical}, we have
\begin{equation*}
\expect \left (\sup_{k \in \mathcal{K}} \sup_{r > 0} \sup_{f_1, f_2 \in r B_k} \lVert k \rVert_{\diag}^{-1/2} \left \lvert \frac{1}{n} \sum_{i = 1}^n \eps_i \frac{(r^{-1/2} V f_1 (X_i) - r^{-1/2} V f_2 (X_i))^2}{4 C} \right \rvert \; \middle \vert X \right )
\end{equation*}
is at most
\begin{equation*}
2 \expect \left (\sup_{k \in \mathcal{K}} \sup_{r > 0} \sup_{f_1, f_2 \in r B_k} \lVert k \rVert_{\diag}^{-1/2} \left \lvert \frac{1}{n} \sum_{i = 1}^n \eps_i (r^{-1} V f_1 (X_i) - r^{-1} V f_2 (X_i)) \right \rvert \; \middle \vert X \right )
\end{equation*}
almost surely. Therefore,
\begin{align*}
\expect(Z) &\leq 16 C \expect \left (\sup_{k \in \mathcal{K}} \sup_{r > 0} \sup_{f_1, f_2 \in r B_k} \lVert k \rVert_{\diag}^{-1/2} \left \lvert \frac{1}{n} \sum_{i = 1}^n \eps_i (r^{-1} V f_1 (X_i) - r^{-1} V f_2 (X_i)) \right \rvert \right ) \\
&\leq 32 C \expect \left (\sup_{k \in \mathcal{K}} \sup_{r > 0} \sup_{f \in r B_k} \lVert k \rVert_{\diag}^{-1/2} \left \lvert \frac{1}{n} \sum_{i = 1}^n \eps_i r^{-1} V f(X_i) \right \rvert \right )
\end{align*}
by the triangle inequality. Again, by Theorem 3.2.1 of \cite{gine2015mathematical}, we have
\begin{equation*}
\expect(Z) \leq 64 C \expect \left (\sup_{k \in \mathcal{K}} \sup_{r > 0} \sup_{f \in r B_k} \lVert k \rVert_{\diag}^{-1/2} \left \lvert \frac{1}{n} \sum_{i = 1}^n \eps_i r^{-1} f(X_i) \right \rvert \right )
\end{equation*}
since $V$ is a contraction vanishing at 0. We have
\begin{align*}
&\sup_{k \in \mathcal{K}} \sup_{r > 0} \sup_{f \in r B_k} \lVert k \rVert_{\diag}^{-1/2} \left \lvert \frac{1}{n} \sum_{i = 1}^n \eps_i r^{-1} f(X_i) \right \rvert \\
= \; &\sup_{k \in \mathcal{K}} \sup_{r > 0} \sup_{f \in r B_k} \lVert k \rVert_{\diag}^{-1/2} \left \lvert \left \langle \frac{1}{n} \sum_{i = 1}^n \eps_i k_{X_i}, r^{-1} f \right \rangle_k \right \rvert \\
= \; &\sup_{k \in \mathcal{K}} \lVert k \rVert_{\diag}^{-1/2} \left \lVert \frac{1}{n} \sum_{i = 1}^n \eps_i k_{X_i} \right \rVert_k \\
= \; &\sup_{k \in \mathcal{K}} \lVert k \rVert_{\diag}^{-1/2} \left ( \frac{1}{n^2} \sum_{i, j = 1}^n \eps_i \eps_j k(X_i, X_j) \right )^{1/2}
\end{align*}
by the reproducing kernel property and the Cauchy--Schwarz inequality. By Lemma \ref{lKernChain} with $\sigma^2 = 1$, Lemma \ref{lPsi1} and Jensen's inequality, we have $\expect(Z) \leq 107 J C n^{-1/2}$.

Let
\begin{equation*}
A = \left \{ \lVert k \rVert_{\diag}^{-1/4} r^{-1/2} V f_1 - \lVert k \rVert_{\diag}^{-1/4} r^{-1/2} V f_2 : k \in \mathcal{K}, r > 0 \text{ and } f_1, f_2 \in r B_k \right \}.
\end{equation*}
We have that $A \subs C(S)$ is separable by Lemma \ref{lSep}. Furthermore,
\begin{align*}
\left \lVert \lVert k \rVert_{\diag}^{-1/4} r^{-1/2} V f_1 - \lVert k \rVert_{\diag}^{-1/4} r^{-1/2} V f_2 \right \rVert_\infty &\leq \min \left ( 2 C \lVert k \rVert_{\diag}^{-1/4} r^{-1/2}, 2 \lVert k \rVert_{\diag}^{1/4} r^{1/2} \right ) \\
&\leq 2 C^{1/2}
\end{align*}
for all $k \in \mathcal{K}$, all $r > 0$ and all $f_1, f_2 \in r B_k$. By Lemma \ref{lTal}, we have
\begin{equation*}
Z \leq \expect(Z) + \left ( \frac{32 C^2 t}{n} + \frac{16 C \expect(Z) t}{n} \right )^{1/2} + \frac{8 C t}{3 n}
\end{equation*}
with probability at least $1 - e^{-t}$. We have $\expect(Z) \leq 107 J C n^{-1/2}$ from above. The result follows. 
\end{proof}

We move the bound on the distance between $V \hat h_{k, r}$ and $V h_{k, r}$ from the $L^2(P_n)$ norm to the $L^2 (P)$ norm for $k \in \mathcal{K}$, $r \geq 0$ and $h_{k, r} \in r B_k$.

\begin{corollary} \label{cVaryTrueNorm}
Assume (Y) and ($\mathcal{K}1$). Let $t \geq 1$ and recall the definitions of $A_{3, t}$ and $A_{4, t}$ from Lemmas \ref{lVaryBias} and \ref{lVarySwapNorm}. On the set $A_{3, t} \cap A_{4, t} \in \mathcal{F}$, for which $\prob(A_{3, t} \cap A_{4, t}) \geq 1 - 2 e^{-t}$, we have
\begin{equation*}
\lVert V \hat h_{k, r} - V h_{k, r} \rVert_{L^2 (P)}^2 \leq \frac{J \lVert k \rVert_{\diag}^{1/2} (151 C + 21 \sigma) r t^{1/2}}{n^{1/2}} + \frac{8 \lVert k \rVert_{\diag}^{1/2} C r t}{3 n} + 4 \lVert h_{k, r} - g \rVert_\infty^2
\end{equation*}
simultaneously for all $k \in \mathcal{K}$, all $r \geq 0$ and all $h_{k, r} \in r B_k$.
\end{corollary}

\begin{proof}
By Lemma \ref{lVaryBias}, we have
\begin{equation*}
\lVert \hat h_{k, r} - h_{k, r} \rVert_{L^2 (P_n)}^2 \leq \frac{21 J \lVert k \rVert_{\diag}^{1/2} \sigma r t^{1/2}}{n^{1/2}} + 4 \lVert h_{k, r} - g \rVert_\infty^2
\end{equation*}
for all $k \in \mathcal{K}$, all $r \geq 0$ and all $h_{k, r} \in r B_k$, so
\begin{equation*}
\lVert V \hat h_{k, r} - V h_{k, r} \rVert_{L^2 (P_n)}^2 \leq \frac{21 J \lVert k \rVert_{\diag}^{1/2} \sigma r t^{1/2}}{n^{1/2}} + 4 \lVert h_{k, r} - g \rVert_\infty^2.
\end{equation*}
Since $\hat h_{k, r}, h_{k, r} \in r B_k$, by Lemma \ref{lVarySwapNorm} we have
\begin{align*}
&\lVert V \hat h_{k, r} -  V h_{k, r} \rVert_{L^2 (P)}^2 - \lVert V \hat h_{k, r} -  V h_{k, r} \rVert_{L^2 (P_n)}^2 \\
\leq \; &\sup_{f_1, f_2 \in r B_k} \left \lvert \lVert V f_1 - V f_2 \rVert_{L^2 (P_n)}^2 - \lVert V f_1 - V f_2 \rVert_{L^2 (P)}^2 \right \rvert \\
\leq \; &\frac{151 J \lVert k \rVert_{\diag}^{1/2} C r t^{1/2}}{n^{1/2}} + \frac{8 \lVert k \rVert_{\diag}^{1/2} C r t}{3 n}.
\end{align*}
The result follows.
\end{proof}

We assume (g1) to bound the distance between $V \hat h_{k, r}$ and $g$ in the $L^2(P)$ norm for $k \in \mathcal{K}$ and $r \geq 0$ and prove Theorem 6.

{\bf Proof of Theorem 6}
Note that $V g = g$. We have
\begin{align*}
\lVert V \hat h_{k, r} - g \rVert_{L^2 (P)}^2 &\leq \left ( \lVert V \hat h_{k, r} - V h_{k, r} \rVert_{L^2 (P)} + \lVert V h_{k, r} - g \rVert_{L^2 (P)} \right )^2 \\
&\leq 2 \lVert V \hat h_{k, r} - V h_{k, r} \rVert_{L^2 (P)}^2 + 2 \lVert V h_{k, r} - g \rVert_{L^2 (P)}^2 \\
&\leq 2 \lVert V \hat h_{k, r} - V h_{k, r} \rVert_{L^2 (P)}^2 + 2 \lVert h_{k, r} - g \rVert_{L^2 (P)}^2
\end{align*}
for all $k \in \mathcal{K}$, all $r \geq 0$ and all $h_{k, r} \in r B_k$. By Corollary \ref{cVaryTrueNorm}, we have
\begin{equation*}
\lVert V \hat h_{k, r} - V h_{k, r} \rVert_{L^2 (P)}^2 \leq \frac{J \lVert k \rVert_{\diag}^{1/2} (151 C + 21 \sigma) r t^{1/2}}{n^{1/2}} + \frac{8 \lVert k \rVert_{\diag}^{1/2} C r t}{3 n} + 4 \lVert h_{k, r} - g \rVert_\infty^2.
\end{equation*}
Hence,
\begin{equation*}
\lVert V \hat h_{k, r} - g \rVert_{L^2 (P)}^2 \leq \frac{2 J \lVert k \rVert_{\diag}^{1/2} (151 C + 21 \sigma) r t^{1/2}}{n^{1/2}} + \frac{16 \lVert k \rVert_{\diag}^{1/2} C r t}{3 n} + 10 \lVert h_{k, r} - g \rVert_\infty^2.
\end{equation*}
Taking an infimum over $h_{k, r} \in r B_k$ proves the result. \hfill \BlackBox

\section{Covering Numbers for Gaussian Kernels} \label{sCovNum}

Recall that
\begin{equation*}
\mathcal{L} = \left \{ f_\gamma (x_1, x_2) = \exp \left (- \lVert x_1 - x_2 \rVert_2^2 / \gamma^2 \right ) : \gamma \in \Gamma \text{ and } x_1, x_2 \in S \right \} \cup \{0\}.
\end{equation*}
for $\Gamma \subs [u, v]$ non-empty for $v \geq u > 0$. We prove a continuity result about the function $F : \Gamma \to \mathcal{L} \setminus \{0\}$ by $F(\gamma) = f_\gamma$. We also bound the covering numbers of $\mathcal{L}$.

\begin{lemma} \label{lCoverNum}
Assume ($\mathcal{K}2$). Let $\gamma, \eta \in \Gamma$. We have
\begin{equation*}
\lVert f_\gamma - f_\eta \rVert_\infty \leq \frac{(\gamma^2 - \eta^2)^{1/2}}{\gamma \vee \eta}.
\end{equation*}
For $a \in (0, 1)$, we have $N(a, \mathcal{L}, \lVert \cdot \rVert_\infty) \leq \log(v/u)a^{-2} + 2$. For $a \geq 1$, we have $N(a, \mathcal{L}, \lVert \cdot \rVert_\infty) = 1$.
\end{lemma}

\begin{proof}
Let $\gamma \geq \eta$ and $x_1, x_2 \in S$. We have
\begin{align*}
\lvert f_\gamma (x_1, x_2) - f_\eta (x_1, x_2) \rvert &= f_\gamma (x_1, x_2) - f_\eta (x_1, x_2) \\
&\leq \exp \left ( - \lVert x_1 - x_2 \rVert_2^2 / \gamma^2 \right ).
\end{align*}
This is at most $a \in (0, 1)$ whenever $\lVert x_1 - x_2 \rVert_2 > \gamma \log(1/a)^{1/2}$. Suppose $\lVert x_1 - x_2 \rVert_2 \leq \gamma \log(1/a)^{1/2}$. We have
\begin{align*}
\lvert f_\gamma (x_1, x_2) - f_\eta (x_1, x_2) \rvert &= f_\gamma (x_1, x_2) - f_\eta (x_1, x_2) \\
&\leq \exp \left ( \lVert x_1 - x_2 \rVert_2^2 / \eta^2 \right ) (f_\gamma (x_1, x_2) - f_\eta (x_1, x_2)) \\
&= \exp \left ( \lVert x_1 - x_2 \rVert_2^2 \left ( \eta^{-2} - \gamma^{-2} \right ) \right ) - 1 \\
&\leq \exp \left ( \log(1/a) \left ( (\gamma / \eta)^2 - 1 \right ) \right ) - 1.
\end{align*}
This is at most $a$ whenever
\begin{equation}
\gamma \leq \left ( 1 + \frac{\log(1 + a)}{\log(1/a)} \right )^{1/2} \eta. \label{eGaussRat}
\end{equation}
Since $x/(1+x) \leq \log(1 + x) \leq x$ for $x \geq 0$, we have
\begin{align*}
\left ( 1 + \frac{\log(1 + a)}{\log(1/a)} \right )^{1/2} &= \left ( 1 + \frac{\log(1 + a)}{\log(1 + (1-a)/a)} \right )^{1/2} \\
&\geq \left ( 1 + \frac{a/(1+a)}{(1-a)/a} \right )^{1/2} \\
&= \left ( 1 + \frac{a^2}{1 - a^2} \right )^{1/2}.
\end{align*}
Hence, \eqref{eGaussRat} holds whenever
\begin{equation*}
\gamma \leq \left ( 1 + \frac{a^2}{1 - a^2} \right )^{1/2} \eta,
\end{equation*}
or
\begin{equation*}
\log(\gamma) \leq \frac{1}{2} \log \left ( 1 + \frac{a^2}{1 - a^2} \right ) + \log(\eta).
\end{equation*}
The first result follows by rearranging for $a$.

Since
\begin{equation*}
\log \left ( 1 + \frac{a^2}{1 - a^2} \right ) \geq \frac{a^2/(1 - a^2)}{1 + a^2/(1 - a^2)} = a^2,
\end{equation*}
\eqref{eGaussRat} holds whenever $\log(\gamma) \leq a^2 / 2 + \log(\eta)$. Hence, for any $\gamma, \eta \in \Gamma$, we find $\lVert f_\gamma - f_\eta \rVert_\infty \leq a$ whenever $\lvert \log(\gamma) - \log(\eta) \rvert \leq a^2 / 2$. Let $b \geq 1$ and $\gamma_i \in \Gamma$ for $1 \leq i \leq b$. Recall that $\Gamma \subs [u, v]$. If we let
\begin{equation*}
\log(\gamma_i) = \log(u) + a^2 (2 i - 1) / 2
\end{equation*}
and let $b$ be such that
\begin{equation*}
\log(v) - \left ( \log(u) + a^2 (2 b - 1) / 2 \right ) \leq a^2 / 2,
\end{equation*}
then we find the $f_{\gamma_i}$ for $1 \leq i \leq b$ form an $a$ cover of $(\mathcal{L} \setminus \{0\}, \lVert \cdot \rVert_\infty)$. Rearranging the above shows that we can choose
\begin{equation*}
b = \left \lceil \frac{\log(v/u)}{a^2} \right \rceil
\end{equation*}
and the second result follows by adding $\{0\}$ to the cover. The third result follows from the fact that $f_\gamma (x_1, x_2) \in (0, 1]$ for all $\gamma \in \Gamma$ and all $x_1, x_2 \in S$.
\end{proof}

\begin{lemma} \label{lGaussEst}
Assume ($\mathcal{K}2$). We have that $\hat h_{\gamma, r}$ is an $(C(S), \bor(C(S)))$-valued measurable function on $(\Omega \times \Gamma \times [0, \infty), \mathcal{F} \otimes \bor(\Gamma) \otimes \bor([0, \infty)))$, where $\gamma$ varies in $\Gamma$ and $r$ varies in $[0, \infty)$.
\end{lemma}

We calculate an integral of these covering numbers.

\begin{lemma} \label{lCoverInt}
Assume ($\mathcal{K}2$). We have
\begin{equation*}
\int_0^{1/2} \log N(a, \mathcal{L}, \lVert \cdot \rVert_\infty) da \leq \frac{\log(2 + 4 \log(v/u))}{2} + 1.
\end{equation*}
\end{lemma}

\begin{proof}
We have
\begin{equation*}
\int_0^{1/2} \log N(a, \mathcal{L}, \lVert \cdot \rVert_\infty) da \leq \int_0^{1/2} \log \left ( 2 + \log(v/u) a^{-2} \right ) da
\end{equation*}
by Lemma \ref{lCoverNum}. Changing variables to $b = 2 a$ gives
\begin{align*}
\frac{1}{2} \int_0^1 \log \left ( 2 + 4 \log(v/u) b^{-2} \right ) db &\leq \frac{1}{2} \int_0^1 \log \left ( (2 + 4 \log(v/u)) b^{-2} \right ) db \\
&= \frac{\log(2 + 4 \log(v/u))}{2} + \int_0^1 \log(b^{-1}) db.
\end{align*}
Changing variables to $s = \log(b^{-1})$ shows
\begin{equation*}
\int_0^1 \log \left ( b^{-1} \right ) db = \int_0^\infty s \exp(-s) ds = 1
\end{equation*}
since the last integral is the mean of an $\expon(1)$ random variable.
\end{proof}

This lemma allows us directly to gain a bound on $J$.
\begin{lemma} \label{lBoundonJ}
Assume ($\mathcal{K}2$). We have
\begin{equation*}
J \leq \left ( 81 (\log(8 \log(v/u) + 4) + 2) + 1 \right )^{1/2}.
\end{equation*}
\end{lemma}

\section{The Orlicz Space $L^{\psi_1}$} \label{sOrl}

Recall that $\psi_1 (x) = \exp( \lvert x \rvert ) - 1$ for $x \in \real$,
\begin{equation*}
\lVert Z \rVert_{\psi_1} = \inf \{ a \in (0, \infty) : \expect(\psi_1 (Z / a)) \leq 1 \}
\end{equation*}
for any random variable $Z$ on $(\Omega, \mathcal{F})$ and $L^{\psi_1}$ is the set of random variables $Z$ on $(\Omega, \mathcal{F}, \prob)$ such that $\lVert Z \rVert_{\psi_1} < \infty$. We have that $(L^{\psi_1}, \lVert \cdot \rVert_{\psi_1})$ is a Banach space known as an Orlicz space (see \cite{rao1991theory}). For $t \geq 0$, also recall that
\begin{equation*}
\expect(\lvert Z \rvert) \leq (\log 2) \lVert Z \rVert_{\psi_1} \text{ and } \lvert Z \rvert \leq \lVert Z \rVert_{\psi_1} (\log 2 + t)
\end{equation*}
with probability at least $1 - e^{-t}$ by Lemma \ref{lPsi1}. We prove a maximal inequality in $L^{\psi_1}$ using the same method as Lemma 2.3.3 of \cite{gine2015mathematical}.

\begin{lemma} \label{lMaxPsi1}
Let $Z_i \in L^{\psi_1}$ for $1 \leq i \leq I$. Then
\begin{equation*}
\left \lVert \max_{1 \leq i \leq I} \lvert Z_i \rvert \right \rVert_{\psi_1} \leq \frac{\log(2 I)}{\log(5/4)} \max_{1 \leq i \leq I} \lVert Z_i \rVert_{\psi_1}.
\end{equation*}
\end{lemma}

\begin{proof}
Let $M = \max_{1 \leq i \leq I} \lVert Z_i \rVert_{\psi_1}$. Also, let $C \geq 1$ and $a \in (0, \infty)$. By Lemma \ref{lPsi1}, we have
\begin{align*}
\expect \left ( \exp \left ( \max_{1 \leq i \leq I} \lvert Z_i \rvert / a \right ) \right ) &= \int_0^\infty \prob \left ( \max_{1 \leq i \leq I} \lvert Z_i \rvert > a \log t \right ) dt \\
&\leq C + \int_C^\infty \prob \left ( \max_{1 \leq i \leq I} \lvert Z_i \rvert > a \log t \right ) dt \\
&\leq C + \sum_{i = 1}^I \int_C^\infty \prob \left ( \lvert Z_i \rvert > a \log t \right ) dt \\
&\leq C + I \int_C^\infty 2 t^{- a / M} dt.
\end{align*}
Differentiating this bound with respect to $C$ gives $1 - 2 I C^{- a / M}$, so the bound is minimised by $C = (2 I)^{M / a}$. For $a > M$, the bound becomes
\begin{align*}
C + 2 I \frac{M}{a - M} C^{- (a - M)/M} &= (2 I)^{M / a} + \frac{M}{a - M} (2 I)^{1 - (a - M)/a} \\
&= \frac{a}{a - M} (2 I)^{M / a}.
\end{align*}
Let
\begin{equation*}
a = \frac{M \log(2 I)}{\log b}
\end{equation*}
for $b > 1$. We have
\begin{equation*}
\expect \left ( \exp \left ( \max_{1 \leq i \leq I} \lvert Z_i \rvert / a \right ) \right ) \leq 2
\end{equation*}
if $b^2 2^b \leq 4$, the hardest case being $I = 1$. This holds for $b = 5/4$ and the result follows.
\end{proof}
We perform chaining in $L^{\psi_1}$ using the same method as Theorem 2.3.6 of \cite{gine2015mathematical}. Recall that $N(a, M, d)$ is the minimum size of an $a > 0$ cover of a metric space $(M, d)$.

\begin{lemma} \label{lChainPsi1}
Let $Z$ be a stochastic process on $(\Omega, \mathcal{F})$ indexed by a separable metric space $(M, d)$ on which $Z$ is almost-surely continuous with $\lVert Z(s) - Z(t) \rVert_{\psi_1} \leq d(s, t)$ for all $s, t \in M$. Let $D = \sup_{s, t \in M} d(s, t)$. Fix $s_0 \in M$. Then
\begin{equation*}
\left \lVert \sup_{s \in M} \lvert Z(s) - Z(s_0) \rvert \right \rVert_{\psi_1} \leq \frac{4}{\log(5/4)} \int_{0}^{D/2} \log(2 N(a, M, d)) da.
\end{equation*}
\end{lemma}

\begin{proof}
 Since $(M, d)$ is separable, it has a countable dense subset $M_0$. We have
\begin{equation*}
\left \lVert \sup_{s \in M} \lvert Z(s) - Z(s_0) \rvert \right \rVert_{\psi_1} = \left \lVert \sup_{s \in M_0} \lvert Z(s) - Z(s_0) \rvert \right \rVert_{\psi_1}
\end{equation*}
because $Z$ is almost-surely continuous on $M$. Since $M_0$ is countable, there exists a sequence of increasing finite subsets $F_n \subs M$ for $n \geq 1$ whose union is $M_0$. We have
\begin{equation*}
\left \lVert \sup_{s \in M} \lvert Z(s) - Z(s_0) \rvert \right \rVert_{\psi_1} = \lim_{n \to \infty} \left \lVert \max_{s \in F_n} \lvert Z(s) - Z(s_0) \rvert \right \rVert_{\psi_1}
\end{equation*}
by the monotone convergence theorem. Fix $n \geq 1$ and let $F = F_n$. Let $\delta_j = 2^{- j} D$ for $j \geq 0$. Since $F$ is finite, there exists a minimum $J \geq 0$ such that
\begin{equation*}
\{ t \in F : d(s, t) \leq \delta_J \} = \{ s \}
\end{equation*}
for all $s \in F$. Let $A_j$ for $0 \leq j \leq J-1$ be a $\delta_j$ cover of $(M, d)$ of size $N(\delta_j, M, d)$, where we let $A_0 = \{ s_0 \}$. We define the chain $C : F \times \{ 0, \ldots, J \} \to M$ as follows. Let $C(s, J) = s$ for all $s \in F$. For $1 \leq j \leq J$, given $C(s, j)$, let $C(s, j-1)$ be some closest point in $A_{j-1}$ to $C(s, j)$, depending on $s$ only through $C(s, j)$. We have
\begin{equation*}
Z(s) - Z(s_0) = \sum_{j = 1}^J Z(C(s, j)) - Z(C(s, j-1))
\end{equation*}
for $s \in F$. Hence,
\begin{equation*}
\max_{s \in F} \lvert Z(s) - Z(s_0) \rvert \leq \sum_{j = 1}^J \max_{s \in F} \lvert Z(C(s, j)) - Z(C(s, j-1)) \rvert.
\end{equation*}
By Lemma \ref{lMaxPsi1}, we have
\begin{align*}
\left \lVert \max_{s \in F} \lvert Z(s) - Z(s_0) \rvert \right \rVert_{\psi_1} &\leq \sum_{j = 1}^J \left \lVert \max_{s \in F} \lvert Z(C(s, j)) - Z(C(s, j-1)) \rvert \right \rVert_{\psi_1} \\
&\leq \sum_{j = 1}^J \frac{\log(2 N(\delta_j, M, d)) \delta_{j-1}}{\log(5/4)} \\
&= \frac{4}{\log(5/4)} \sum_{j = 1}^J (\delta_j - \delta_{j+1}) \log(2 N(\delta_j, M, d)) \\
&\leq \frac{4}{\log(5/4)} \int_{\delta_{J+1}}^{\delta_1} \log(2 N(a, M, d)) da \\
&\leq \frac{4}{\log(5/4)} \int_{0}^{D/2} \log(2 N(a, M, d)) da.
\end{align*}
The result follows.
\end{proof}

\section{Subgaussian Random Variables and Symmetric Matrices} \label{sSGRV}

The following result is Lemma 31 of \cite{page2017ivanov}, which is essentially Theorem 2.1 from \cite{quintana2014measurable}.

\begin{lemma} \label{lDiag}
Let $M$ be a non-negative-definite matrix which is an $(\real^{n \times n}, \bor(\real^{n \times n}))$-valued measurable matrix on $(\Omega, \mathcal{F})$. There exist an orthogonal matrix $A$ and a diagonal matrix $D$ which are both $(\real^{n \times n}, \bor(\real^{n \times n}))$-valued measurable matrices on $(\Omega, \mathcal{F})$ such that $M = A D A^{\trans}$.
\end{lemma}

 Recall that for $m \times n$ matrices $U$ and $V$, we define $U \circ V$ to be the $m \times n$ matrix with
\begin{equation*}
(U \circ V)_{i, j} = U_{i, j} V_{i, j}.
\end{equation*}
The following lemma is a conditional version of Theorem 1.1 of \cite{rudelson2013hanson}, but with explicit values for the constants derived here.

\begin{lemma} \label{lQuadMGF}
Let $\eps_i$ for $1 \leq i \leq n$ be random variables on $(\Omega, \mathcal{F}, \prob)$ which are independent conditional on some sub-$\sigma$-algebra $\mathcal{G} \subs \mathcal{F}$ and let
\begin{equation*}
\expect( \exp(t \eps_i) \vert \mathcal{G}) \leq \exp( \sigma^2 t^2 / 2)
\end{equation*}
almost surely for $t$ a random variable on $(\Omega, \mathcal{G})$. Let $M$ be an $n \times n$ symmetric matrix which is an $(\real^{n \times n}, \bor(\real^{n \times n}))$-valued measurable matrix on $(\Omega, \mathcal{G})$. We have
\begin{equation*}
\expect \left ( \exp \left ( t \eps^{\trans} (M - I \circ M ) \eps \right ) \middle \vert \mathcal{G} \right ) \leq \exp \left (16 \sigma^4 \tr(M^2) t^2 \right )
\end{equation*}
almost surely for $t$ a random variable on $(\Omega, \mathcal{G})$ such that $32 \sigma^4 \tr(M^2) t^2 \leq 1$.
\end{lemma}

\begin{proof}
We follow the proof of Theorem 1.1 of \cite{rudelson2013hanson}. Let
\begin{equation*}
Z = \eps^{\trans} (M - I \circ M ) \eps = \sum_{i \neq j} M_{i, j} \eps_i \eps_j.
\end{equation*}
Also, let $\phi_i$ for $1 \leq i \leq n$ be random variables on $(\Omega, \mathcal{F}, \prob)$ which are independent of each other, the $\eps_i$ and $\mathcal{G}$, with $\phi_i \sim \bern(1/2)$. Furthermore, let
\begin{equation*}
W = \sum_{i \neq j} \phi_i (1 - \phi_j) M_{i, j} \eps_i \eps_j.
\end{equation*}
We have $Z = 4 \expect(W \vert \mathcal{G}, \eps)$ almost surely, which gives
\begin{equation*}
\exp(t Z) \leq \expect(\exp(4 t W) \vert \mathcal{G}, \eps)
\end{equation*}
almost surely for $t$ a random variable on $(\Omega, \mathcal{G})$ by Jensen's inequality. Let
\begin{equation*}
S = \{ 1 \leq i \leq n : \phi_i = 1 \}.
\end{equation*}
We can write
\begin{equation*}
W = \sum_{i \in S, j \in S^{\cmp}} M_{i, j} \eps_i \eps_j.
\end{equation*}
Since the $\eps_j$ are independent, we have
\begin{align*}
\expect(\exp(t Z) \vert \mathcal{G}) &\leq \expect(\exp(4 t W) \vert \mathcal{G}) \\
&= \expect \left ( \prod_{j \in S^{\cmp}} \expect \left ( \exp \left ( 4 t \sum_{i \in S} M_{i, j} \eps_i \eps_j \right) \middle \vert \mathcal{G}, \phi \right ) \middle \vert \mathcal{G} \right ) \\
&\leq \expect \left ( \prod_{j \in S^{\cmp}} \exp \left ( 8 t^2 \sigma^2 \left ( \sum_{i \in S} M_{i, j} \eps_i \right )^2 \right) \middle \vert \mathcal{G} \right ) \\
&= \expect \left ( \exp \left ( 8 t^2 \sigma^2 \sum_{j \in S^{\cmp}} \left ( \sum_{i \in S} M_{i, j} \eps_i \right )^2 \right) \middle \vert \mathcal{G} \right )
\end{align*}
almost surely. Let $\delta_i$ for $1 \leq i \leq n$ be random variables on $(\Omega, \mathcal{F}, \prob)$ which are independent of each other, the $\eps_i$, the $\phi_i$ and $\mathcal{G}$, with $\delta_i \sim \norm(0, \sigma^2)$. Since the $\eps_i$ are independent, we have
\begin{align*}
\expect(\exp(t Z) \vert \mathcal{G}) &\leq \expect \left ( \exp \left ( 4 t \sum_{j \in S^{\cmp}} \sum_{i \in S} M_{i, j} \eps_i \delta_j \right) \middle \vert \mathcal{G} \right ) \\
&= \expect \left ( \prod_{i \in S} \expect \left ( \exp \left ( 4 t \sum_{j \in S^{\cmp}} M_{i, j} \delta_j \eps_i \right) \middle \vert \mathcal{G}, \phi \right ) \middle \vert \mathcal{G} \right ) \\
&\leq \expect \left ( \prod_{i \in S} \exp \left ( 8 t^2 \sigma^2 \left ( \sum_{j \in S^{\cmp}} M_{i, j} \delta_j \right )^2 \right) \middle \vert \mathcal{G} \right ) \\
&= \expect \left ( \exp \left ( 8 t^2 \sigma^2 \sum_{i \in S} \left ( \sum_{j \in S^{\cmp}} M_{i, j} \delta_j \right )^2 \right) \middle \vert \mathcal{G} \right )
\end{align*}
almost surely. Let $F$ be the $n \times n$ matrix with $F_{i, j} = 1$ if $i = j \in S$ and 0 otherwise. Note that $F$ is an $(\real^{n \times n}, \bor(\real^{n \times n}))$-valued measurable matrix on $(\Omega, \sa(\phi))$. Then
\begin{equation*}
\expect(\exp(t Z) \vert \mathcal{G}) \leq \expect \left ( \exp \left ( 8 t^2 \sigma^2 \delta^{\trans} (I - F) M F M (I - F) \delta \right) \middle \vert \mathcal{G} \right )
\end{equation*}
almost surely. By Lemma \ref{lDiag}, there exist an orthogonal matrix $A$ and a diagonal matrix $D$ which are both $(\real^{n \times n}, \bor(\real^{n \times n}))$-valued measurable matrices on $(\Omega, \sa(\mathcal{G}, \phi))$ such that
\begin{equation*}
(I - F) M F M (I - F) = A D A^{\trans},
\end{equation*}
which is non-negative definite. Since $A^{\trans} \delta$ and $\delta$ have the same distribution given $\mathcal{G}$, we have
\begin{align*}
\expect(\exp(t Z) \vert \mathcal{G}) &\leq \expect \left ( \exp \left ( 8 t^2 \sigma^2 \delta^{\trans} D \delta \right) \middle \vert \mathcal{G} \right ) \\
&= \expect \left ( \prod_{i = 1}^n \expect \left ( \exp \left ( 8 t^2 \sigma^2 D_{i, i} \delta_i^2 \right ) \middle \vert \mathcal{G}, \phi \right ) \middle \vert \mathcal{G} \right ) \\
&= \expect \left ( \prod_{i = 1}^n \left (1 - 16 \sigma^4 D_{i, i} t^2 \right )^{-1/2} \middle \vert \mathcal{G} \right )
\end{align*}
almost surely for $16 \sigma^4 (\max_{1 \leq i \leq n} D_{i, i}) t^2 < 1$ by computing the moment generating function of the $\delta_i^2$. We have that $(1 - x)^{- 1/2} \leq \exp(x)$ for $x \in [0, 1/2]$, so
\begin{equation*}
\expect(\exp(t Z) \vert \mathcal{G}) \leq \expect \left ( \prod_{i = 1}^n \exp \left (16 \sigma^4 D_{i, i} t^2 \right ) \middle \vert \mathcal{G} \right ) = \expect \left ( \exp \left (16 \sigma^4 \tr(D) t^2 \right ) \middle \vert \mathcal{G} \right )
\end{equation*}
almost surely for $32 \sigma^4 (\max_{1 \leq i \leq n} D_{i, i}) t^2 \leq 1$. We have
\begin{equation*}
\tr(D) = \tr( (I - F) M F M (I - F) ) = \sum_{i \in S} \sum_{j \in S^{\cmp}} M_{i, j}^2 \leq \sum_{i  = 1}^n \sum_{j = 1}^n M_{i, j}^2 = \tr(M^2)
\end{equation*}
and
\begin{equation*}
\max_{1 \leq i \leq n} D_{i, i} \leq \tr(D) \leq \tr(M^2).
\end{equation*}
The result follows.
\end{proof}
We move the bound on the conditional moment generating function of $\eps^{\trans} (M - I \circ M ) \eps$ to that of $\lvert \eps^{\trans} (M - I \circ M ) \eps \rvert$.

\begin{lemma} \label{lQuadPsi1}
Let $\eps_i$ for $1 \leq i \leq n$ be random variables on $(\Omega, \mathcal{F}, \prob)$ which are independent conditional on some sub-$\sigma$-algebra $\mathcal{G} \subs \mathcal{F}$ and let
\begin{equation*}
\expect( \exp(t \eps_i) \vert \mathcal{G}) \leq \exp( \sigma^2 t^2 / 2)
\end{equation*}
almost surely for $t$ a random variable on $(\Omega, \mathcal{G})$. Let $M$ be an $n \times n$ symmetric matrix which is an $(\real^{n \times n}, \bor(\real^{n \times n}))$-valued measurable matrix on $(\Omega, \mathcal{G})$. We have
\begin{equation*}
\expect \left ( \exp \left ( t \rvert \eps^{\trans} (M - I \circ M ) \eps \rvert \right ) \middle \vert \mathcal{G} \right ) \leq \frac{1}{1 - 2^{7/2} \sigma^2 \tr(M^2)^{1/2} t} 2^{2^{7/2} \sigma^2 \tr(M^2)^{1/2} t}
\end{equation*}
almost surely for $t \geq 0$, a random variable on $(\Omega, \mathcal{G})$, such that $2^{7/2} \sigma^2 \tr(M^2)^{1/2} t < 1$. Hence,
\begin{equation*}
\expect \left ( \frac{\left \lvert \eps^{\trans} (M - I \circ M ) \eps \right \rvert}{2^{7/2} (\log 2) \sigma^2 \tr(M^2)^{1/2}/\log(5/4)} \middle \vert \mathcal{G} \right ) \leq 2.
\end{equation*}
\end{lemma}

\begin{proof}
Let
\begin{equation*}
Z = \eps^{\trans} (M - I \circ M ) \eps.
\end{equation*}
By Lemma \ref{lQuadMGF}, we have
\begin{equation*}
\expect(\exp(t Z) \vert \mathcal{G}) \leq \exp \left (16 \sigma^4 \tr(M^2) t^2 \right )
\end{equation*}
almost surely for $t$ a random variable on $(\Omega, \mathcal{G})$ such that $32 \sigma^4 \tr(A^2) t^2 \leq 1$. By Chernoff bounding, we have
\begin{equation*}
\prob(Z \geq z \vert \mathcal{G}) \leq \exp \left ( -t z + 16 \sigma^4 \tr(M^2) t^2 \right )
\end{equation*}
almost surely for $z \geq 0$, a random variable on $(\Omega, \mathcal{G})$, $t \geq 0$ and $32 \sigma^4 \tr(A^2) t^2 \leq 1$. Minimising over $t$ gives
\begin{equation*}
\prob(Z \geq z \vert \mathcal{G}) \leq \exp \left ( - \min \left ( \frac{z^2}{2^6 \sigma^4 \tr(M^2)}, \frac{z}{2^{7/2} \sigma^2 \tr(M^2)^{1/2}} \right ) \right )
\end{equation*}
almost surely. The first term in the minimum is attained by $t = 2^{-5} \sigma^{-4} \tr(M^2)^{-1} z$ when $z < 2^{5/2} \sigma^2 \tr(M^2)^{1/2}$, and the second term is attained by $t = 2^{-5/2} \sigma^{-2} \tr(M^2)^{-1/2}$ when $z \geq 2^{5/2} \sigma^2 \tr(M^2)^{1/2}$. In the second case, note that
\begin{equation*}
16 \sigma^4 \tr(M^2) t^2 = \frac{1}{2} \leq \frac{z}{2^{7/2} \sigma^2 \tr(M^2)^{1/2}}.
\end{equation*}
 The same result holds if we replace $Z$ with $- Z$ by replacing $M$ with $- M$. For $C \geq 1$ and $t \geq 0$, random variables on $(\Omega, \mathcal{G})$, we have
\begin{align*}
\expect(\exp(t \lvert Z \rvert) \vert \mathcal{G}) &= \int_0^\infty \prob( \lvert Z \rvert \geq (\log s) / t \vert \mathcal{G}) ds \\
&\leq C + \int_C^\infty \prob( \lvert Z \rvert \geq (\log s) / t \vert \mathcal{G}) ds \\
&\leq C + \int_C^\infty \prob(Z \geq (\log s) / t \vert \mathcal{G}) ds + \int_C^\infty \prob(- Z \geq (\log s) / t \vert \mathcal{G}) ds \\
&\leq C + 2 \int_C^\infty \exp \left ( - \min \left ( \frac{(\log s)^2}{2^6 \sigma^4 \tr(M^2) t^2}, \frac{\log s}{2^{7/2} \sigma^2 \tr(M^2)^{1/2} t} \right ) \right ) ds
\end{align*}
almost surely. By letting $C \geq \exp(2^{5/2} \sigma^2 \tr(M^2)^{1/2} t)$, the bound becomes
\begin{equation*}
C + 2 \int_C^\infty s^{- (2^{7/2} \sigma^2 \tr(M^2)^{1/2} t)^{-1}} ds.
\end{equation*}
Let $u = 2^{7/2} \sigma^2 \tr(M^2)^{1/2} t$, a random variable on $(\Omega, \mathcal{G})$. Differentiating this bound with respect to $C$ gives $1 - 2 C^{- u^{-1}}$, so the bound is minimised by $C = 2^u$. This satisfies the condition on $C$ above as
\begin{equation*}
e^{2^{5/2}} \leq 3^6 \leq 2^{10} \leq 2^{2^{7/2}}.
\end{equation*}
For $u < 1$, the bound becomes
\begin{align*}
C + 2 \frac{u}{1 - u} C^{- (1 - u)/u} &= 2^u + \frac{u}{1 - u} 2^{1 - (1 - u)} \\
&= \frac{1}{1 - u} 2^u.
\end{align*}
The first result follows. Let
\begin{equation*}
u = \frac{\log b}{\log 2}
\end{equation*}
for $b > 1$. We have
\begin{equation*}
\expect(\exp(t \lvert Z \rvert) \vert \mathcal{G}) \leq 2
\end{equation*}
almost surely if $b^2 2^b \leq 4$. This holds for $b = 5/4$ and the second result follows.
\end{proof}

\bibliographystyle{plainnat}
\bibliography{paper}